\documentclass[12pt,reqno]{amsart}
\usepackage{multicol,amsmath,graphics,cancel,soul,color,bbm,amsfonts,dsfont,tikz,mathrsfs}
\usepackage[left=1in, right=1in, top=1.1in,bottom=1.1in]{geometry}
\usetikzlibrary{matrix,patterns,positioning}
\numberwithin{equation}{section}

\newcommand{\Cc}{\mathcal{C}}

\newcommand{\Hc}{\mathcal{H}}

\newcommand{\Oc}{\mathcal{O}}
\newcommand{\Pc}{\mathcal{P}}

\newcommand{\Sc}{\mathcal{S}}

\newcommand{\C}{\mathbb{C}}

\newcommand{\E}{\mathbb{E}}

\newcommand{\N}{\mathbb{N}}
\newcommand{\Pb}{\mathbb{P}}

\newcommand{\R}{\mathbb{R}}

\newcommand{\Norm}[1]{\left\lVert#1\right\rVert}
\newcommand{\Abs}[1]{\left|#1\right|}
\newcommand{\Ip}[1]{\left\langle #1 \right\rangle}

\newcommand{\Indi}[1]{\mathbbm{1}_{#1}}

\newtheorem{thm}{Theorem}[section]
\newtheorem{lemma}[thm]{Lemma}

\newcounter{dummy} \numberwithin{dummy}{section}
\newtheorem{Definition}[dummy]{Definition}
\newtheorem{Proposition}[dummy]{Proposition}
\newtheorem{Theorem}[dummy]{Theorem}
\newtheorem{Corollary}[dummy]{Corollary}
\newtheorem{Lemma}[dummy]{Lemma}

\newtheorem{Remark}[dummy]{Remark}

\def\1{{\rm l}\hskip -0.21truecm 1}

\begin{document}
\title[Collision of eigenvalues for matrix-valued processes]{Collision of eigenvalues for matrix-valued processes}
 \date{\today}

\author{Arturo Jaramillo}
\address{Arturo Jaramillo: Department of Mathematics, University of Kansas,   Lawrence, KS 66045, USA.}
\email{jagil@ku.edu}
\author{David Nualart} \thanks{D. Nualart is supported by the NSF grant DMS15112891}
\address{David Nualart: Department of Mathematics, University of Kansas,   Lawrence, KS 66045, USA.}
\email{nualart@ku.edu}
\date{\today}
\begin{abstract}
We examine the probability that at least two eigenvalues of an Hermitian matrix-valued Gaussian process, collide. In particular, we determine sharp conditions under which such probability is zero. 
As an application, we show that the eigenvalues of a real symmetric matrix-valued fractional Brownian motion of Hurst parameter $H$, collide when $H<1/2$ and don't collide when $H>\frac{1}{2}$, while those of a complex Hermitian fractional Brownian motion collide when $H<\frac{1}{3}$ and don't collide when $H>\frac{1}{3}$. Our approach is based on the relation between hitting probabilities for Gaussian processes with the capacity and Hausdorff dimension of measurable sets.
\end{abstract}
\maketitle

\section{Introduction}
  For $r\in\N$ fixed, consider  a centered Gaussian  random field  $\xi = \{\xi(t); t\in \R_{+}^r\}$, defined in a probability space $(\Omega, \mathcal{F}, \mathbb{P})$, with covariance function
given by  
\begin{align*}
\E\big[\xi(s)\xi(t)\big]
=R(s,t),
\end{align*}
for some non-negative definite function $R:(\R_+^r)^2\rightarrow\R$.
 Let $\{ \xi_{i,j}, \eta_{i,j} ; i,j\in\N\}$,  be a family of independent copies of $\xi$. For $\beta\in\{1,2\}$ and $d\in\N$, with $d\geq 2$ fixed, consider the matrix-valued process $X^{\beta}=\{X_{i,j}^{\beta}(t) ; t\in\R_{+}^{r},\ 1\leq i,j\leq d\}$, defined by
\begin{align}\label{eq:Xdef}
X_{i,j}^{\beta}(t)
  &=\left\{\begin{array}{ll}\xi_{i,j}(t)+\textbf{i}\Indi{\{\beta=2\}}\eta_{i,j}(t)&\text{ if }\ i<j\\(\Indi{\{\beta=1\}}\sqrt{2}+\Indi{\{\beta=2\}})\xi_{i,i}(t)+\textbf{i}\Indi{\{\beta=2\}}\eta_{i,i}(t)&\text{ if }\ i=j\\\xi_{i,j}(t)-\textbf{i}\Indi{\{\beta=2\}}\eta_{i,j}(t)&\text{ if }\ j<i.\end{array}\right.
\end{align}
In accordance to the type of symmetry of $X^{\beta}(t)$, we will refer to $X^{1}$ and $X^{2}$ as the Gaussian orthogonal ensemble process (GOE) and Gaussian unitary ensemble process (GUE), respectively. Let $A^\beta$ be a fixed Hermitian deterministic matrix, such that $A^\beta$ has real entries in the case $\beta=1$, and complex entries in the case $\beta=2$.

Consider the set  of the ordered eigenvalues $\lambda_{1}^{\beta}(t)\geq\cdots\geq\lambda^{\beta}_{d}(t)$ of 
\begin{equation} \label{ybeta}
Y^{\beta}(t):=A^{\beta}+X^{\beta}(t).
\end{equation}
The purpose of this paper is to determine necessary and sufficient conditions under which, with probability one, we have $\lambda_{1}^{\beta}(t)>\cdots>\lambda_{d}^{\beta}(t)$ for all $t$ belonging to a suitable rectangle of $\R_+^{r}$.

The matrix-valued process $Y^{\beta}$ was first studied by Dyson for $\beta=r=1$, in the case where $\xi$ is a standard Brownian. In particular, he proved that the processes $\lambda_{1}^{1},\dots, \lambda_{d}^{1}$ satisfy a system of stochastic differential equations with non-smooth diffusion coefficients, as well as the non-collision property
\begin{align}\label{eq:noncollision0}
\Pb\left[\lambda_{i}^1(t)=\lambda_{j}^1(t)\ \ \text{for some }\ t>0\ \text{and}\ 1\leq i<j\leq n\right]=0.
\end{align}
For a more recent treatment of these results, see \cite[Section~4.3]{AnGuZe}.

Afterwards, Nualart and P\'erez-Abreu used Young's theory of integration, to prove that in the case where $\beta=r=1$ and $\xi$ is a Gaussian process with H\"older continuous parths larger than $\frac{1}{2}$, relation \eqref{eq:noncollision0} holds. This result can be applied to the case where $X^1$ is a fractional Brownian matrix with Hurst parameter $\frac{1}{2}<H<1$. Namely, when $\xi=\{ \xi(t) ; t\ge 0\}$ is  centered Gaussian processes with  covariance
\begin{align}\label{eq:covfbm}
R(s,t)
  &=\frac{1}{2}(t^{2H}+s^{2H}-\Abs{t-s}^{2H}).
\end{align}
In this manuscript we prove, among other things, that the results presented in \cite{NuPe} are sharp, in the sense that for $H<1/2$, the eigenvalues  $\lambda_1^1,\dots, \lambda_{d}^{1}$ collide with positive probability, and with probability one if $A^1=0$. We also give an alternative proof of the results obtained by Nualart and P\'erez-Abreu in \cite{NuPe}. In the Brownian motion case $H=\frac{1}{2}$, the method we apply reduces to the one presented in Section 4.9 of the book \cite{McKean} by Mckean. On the other hand, we obtain the surprising results that for the fractional Hermitian matrix $X^2$, the eigenvalues   $\lambda_1^2,\dots, \lambda_{d}^{2}$ do not collide when $H>\frac 13$ and collide with positive probability (or with probability one if $A^2=0$), when $H<\frac 13$. The case $H=\frac 13$ cannot be handled with the techniques used in this paper  and remains an open problem.

  When $\psi(s,t)$ is of the form \eqref{eq:covfbm} and $\beta=1$, the non-collision property is of great interest, since it is a necessary condition for characterizing $(\lambda_{1}^1,\dots,\lambda_{d}^1)$ as the unique solution of a Young integral equation (in the case where $H>\frac{1}{2}$), or as an It\^o stochastic differential equation (in the case $H=\frac{1}{2}$). We refer the reader to \cite{AnGuZe} and \cite{PaPeGa} for a proof of such characterizations.

 The goal of this manuscript is to investigate the probability of collision of the eigenvalues $\lambda_{1}^\beta,\dots, \lambda_{d}^\beta$, for $\xi$ belonging to a class of processes that includes the complex Hermitian and real symmmetric fractional Brownian motion of Hurst parameter $H\neq \frac{1}{2}$. The proofs of our main results are based on estimations of hitting probabilities for Gaussian processes, as well as some geometric properties of the set of degenerate matrices. This approach is different from the  methodology used in \cite{NuPe} and \cite{AnGuZe}, where the process $(\lambda_{1}^1,\dots,\lambda_{d}^1)$ is studied by means of stochastic integral techniques. 

\section{Main results}\label{eq:mainresults}
As mentioned before, the ideas presented in this manuscript rely heavily on the the hitting probability estimations presented in \cite{BiLaXi}. In order to apply such results, we will assume that the there exists a multiparameter index $(H_{1},\dots, H_{r})\in (0,1)^{r}$, and an interval 
\begin{align}\label{eq:Idef}
I=[a,b]:=\prod_{j=1}^{r}[a_{j},b_{j}]\subset\R_{+}^{r},
\end{align}
with $a=(a_{1},\dots, a_{r}),b=(b_{1},\dots, b_{r})\in\R_{+}^{r}$ satisfying $a_{i}\leq b_{i}$ for $1\leq i\leq r$, such that the following technical conditions hold:\\

\noindent\textbf{(H1)} There exist strictly positive and finite constants $c_{2,1},c_{2,2}$ and $c_{2,3}$ such that $\E\left[\xi(t)^{2}\right]\geq c_{2,1}$ for all $t\in I$ and 
\begin{align*}
c_{2,2}\sum_{j=1}^{r}\Abs{s_{j}-t_{j}}^{2H_{j}}\leq\E\left[(\xi(s)-\xi(t))^2\right]
\leq c_{2,3}\sum_{j=1}^{r}\Abs{s_{j}-t_{j}}^{2H_{j}},
\end{align*}
for $s,t\in I$ of the form $s=(s_{1},\dots, s_{r})$ and $t=(t_{1},\dots, t_{r})$.\\

\noindent \textbf{(H2)} There exists a constant $c_{2,4}>0$ such that for all $s=(s_{1},\dots, s_r),t=(t_{1},\dots, t_{r})\in I,$
\begin{align*}
\textup{Var}\left[\xi(t)\ |\ \xi(s)\right]
  &\geq c_{2,4}\sum_{j=1}^{r}\Abs{s_{j}-t_{j}}^{2H_{j}},
\end{align*}
where $\textup{Var}\left[\xi(t)\ |\ \xi(s)\right]$ denotes the conditional variance of $\xi(t)$ given $\xi(s)$.\\

 The collection of random fields satisfying conditions \textbf{(H1)} and \textbf{(H2)} includes, among others, the fractional Brownian sheet and the solutions to the stochastic heat equation driven by space-time white noise. Our main results are Theorem \ref{thm:main} and Corollary \ref{cor:main} below. The proofs will be presented in Section \ref{sec:proofs}.
\begin{Theorem}\label{thm:main}
Define $Q:=\sum_{j=1}^{r}\frac{1}{H_{j}}$. Then, for $\beta=1,2$, we have the following results:
\begin{enumerate}
\item[(i)] If $Q<\beta+1$,  
\begin{align}\label{eq:main100i}
\Pb\left[\lambda_{i}^{\beta}(t)=\lambda_{j}^{\beta}(t)\ \ \text{for some }\ t\in I\ \text{and}\ 1\leq i<j\leq n\right]=0.
\end{align}
\item[(ii)] If $Q>\beta+1$,  
\begin{align}\label{eq:main100}
\Pb\left[\lambda_{i}^{\beta}(t)=\lambda_{j}^{\beta}(t)\ \ \text{for some }\ t\in I\ \text{and}\ 1\leq i<j\leq n\right]>0.
\end{align}
\end{enumerate}
\end{Theorem}

In the case where $\xi$ is a one-parameter fractional Brownian motion with Hurst parameter $H\in(0,1)$, we prove a generalization of Theorem \ref{thm:main}, where in addition to characterizing the eigenvalue collision property in terms of $H$, for $H\neq \frac{1}{2}$, we provide  conditions under which such collision occurs instantly (see equation \eqref{eq:fbmprobone} below).
\begin{Corollary}\label{cor:main}
If $\xi=\{ \xi(t) ; t\ge 0\}$ is a fractional Brownian motion with Hurst parameter   $0<H<1$ and  $I= [a,b]$, where $0<a<b$.  we have the following  results:
\begin{enumerate}
\item[(i)] If $\frac{1}{1+\beta}<H<1$,
\begin{align}\label{eq:main10}
\Pb\left[\lambda_{i}^{\beta}(t)=\lambda_{j}^{\beta}(t)\ \ \text{for some }\ t\in I\ \text{and}\ 1\leq i,j\leq n\right]=0.
\end{align}
\item[(ii)] If $0<H<\frac{1}{1+\beta}$,
\begin{align}\label{eq:main1}
\Pb\left[\lambda_{i}^{\beta}(t)=\lambda_{j}^{\beta}(t)\ \ \text{for some }\ t\in I\ \text{and}\ 1\leq i,j\leq n\right]>0.
\end{align}
Moreover, if either $A^{\beta}=0$ or the spectrum of $A^{\beta}$ has cardinality $d-1$, then 
\begin{align}\label{eq:fbmprobone}
\Pb\left[\lambda_{i}^{\beta}(t)=\lambda_{j}^{\beta}(t)\ \ \text{for some }\ t >0\ \text{and}\ 1\leq i,j\leq n\right]=1.
\end{align}
\end{enumerate}
\end{Corollary}

\begin{Remark}
Combining Corollary \ref{cor:main} with \cite[Section~4.3]{AnGuZe}, we conclude that the condition $H\geq \frac{1}{2}$ is   necessary and sufficient  for the non-collision property of real symmetric fractional Brownian matrices.  On the other hand, the critical value for the collision property for the fractional  GUE is $H=\frac 13$.
Nevertheless, our proof of Corollary \ref{cor:main} is not valid for the critical value $H=\frac{1}{1+\beta}$. Thus, if $\beta=2$ and $H=\frac{1}{3}$, the non-collision property for $\lambda_{1}^{2},\dots, \lambda_{d}^{2}$ is still an open problem. 
\end{Remark}
The rest of the paper is organized as follows. Section 3 contains the results from hitting probabilities for Gaussian fields that we will use throughout the paper. In Section 4, we describe some geometric properties of the set of degenerate Hermitian matrices of dimension $d$; namely, the Hermitian  matrices with at least one repeated eigenvalue. Finally, in Section 5 we prove Theorem \ref{cor:main} and Corollary \ref{cor:main}.

\section{Hitting probabilities}\label{sec:preliminaries}
 In this section we present some results on hitting probabilities for Gaussian fields and their relation to the capacity and Hausdorff dimension of Borel sets. We will closely follow the work by Bierm\'e, Lacaux and Xiao presented in \cite{BiLaXi}, and we refer the interested reader to \cite{BiLaXi, Xiao, XiaoPDE} for a detailed treatment of the theory of hitting probabilities.

  For $n\in\N$, let $W=\{(W_{1}(t),\dots,W_{n}(t)) ; t\in\R_{+}^{r}\}$ be an $n$-dimensional Gaussian field, whose entries are independent copies of $\xi$. In the sequel, for every $q>0$ and any Borel set $F\subset\R^{n}$, $\mathcal{H}_{q}(F)$ will denote the $q$-dimensional Hausdorff measure of $F$ and $\Cc_{\alpha}(F)$ will denote the Bessel-Riesz capacity of order $\alpha$ of $F$, defined by 
\begin{align}\label{eq:capacitydef}
\Cc_{\alpha}(F)
  &:=\bigg(\inf_{\mu\in\Pc(F)}\int_{\R^{n}}\int_{\R^{n}}f_{\alpha}(\|x-y\|)\mu(dx)\mu(dy)\bigg)^{-1},
\end{align}
where $\Pc(F)$ is the family of probability measures supported in $F$ and the function $f_{\alpha}:\R_{+}\rightarrow\R_{+}$ is defined by 
\begin{align}\label{eq:falphadef}
f_{\alpha}(r)
  &:=\left\{\begin{array}{ll}r^{-\alpha}   &\ \text{ if } \alpha>0,\\\log\big(\frac{e}{r\wedge 1}\big)   &\ \text{ if } \alpha=0,\\1   &\ \text{ if } \alpha<0.\end{array}\right.
\end{align}
Define as well the Hausdorff dimension ${\rm dim}_{H}(F)$, by
$$
{\rm dim}_{H}(F):=\inf\{q>0\ |\ \mathcal{H}_{q}(F)=0\}.
$$
We refer the reader to \cite{Falc,Kah} for basic properties of the Hausdorff measure and capacity of Borel sets. The following results, presented in \cite[Theorem~2.1]{BiLaXi}, will be used to prove Theorem \ref{thm:main}.

\begin{Theorem}[Bierm\'e, Lacaux and Xiao]  \label{BiLaXicapacity} 
Consider an interval $I$ of the form \eqref{eq:Idef}. If $F\subset\R^{n}$ is a Borel set, then there exist constants $c_1,c_2>0$, such that
\begin{align*}
c_1\Cc_{n-Q}(F)
  &\leq \Pb\left[W^{-1}(F)\cap I\neq \emptyset\right]
	\leq c_2\mathcal{H}_{n-Q}(F),
\end{align*}
where $Q=\sum_{j=1}^r\frac{1}{H_{j}}$.
\end{Theorem}
 As a consequence, we have the following result.
\begin{Corollary}\label{cor:BiLaXi}
Let $F\subset\R^{n}$ be a Borel set. Then 
\begin{enumerate}
\item If ${\rm dim}_{H}(F)<n-Q$, the set $W^{-1}(F)\cap I$ is empty with probability one.
\item If ${\rm dim}_{H}(F)>n-Q$, the set $W^{-1}(F)\cap I$ is non-empty with positive probability.
\end{enumerate}
\end{Corollary}

\section{Geometric properties of degenerate Hermitian matrices}\label{sec:geo}
 
Let $\Sc(d)$ and $\Hc(d)$ denote the set of real symmetric matrices and complex Hermitian matrices, respectively.  Define 
\begin{align*}
n_{\beta}(d)
  &:=\left\{\begin{array}{cc}d(d+1)/2  &\text{ if }\ \beta=1\\d^2  &\text{ if }\ \beta=2.\end{array}\right.
\end{align*}
In the sequel, we will identify a given element $x\in\R^{n_1(d)}$ with the unique $\hat{x}=\{\hat{x}_{i,j}\}_{1\leq i,j\leq d}\in\Sc(d)$ satisfying $\hat{x}_{i,j}=x_{\frac{1}{2}i(1+2d-i)-d+j}$, for $1\leq i\leq j\leq d$. In a similar way, we can identify an element $x\in\R^{n_{2}(d)}$ with the unique $\hat{x}\in\Hc(d)$ given by 
 \begin{align*}
\hat{x}_{i,j}
  &=\left\{\begin{array}{cc}x_{\frac{1}{2}i(1+2 d-i)-d} &\text{ if } i=j\\
	x_{\frac{1}{2}i(1+2 d-i)-d+j}+\textbf{i}x_{n_{1}(d)+\frac{1}{2}i(2 d-i-1)-d+j}  &\text{ if } i<j.\end{array}\right.
\end{align*}
We will denote by $\Phi_{i}(x)$ the $i$-th largest eigenvalue of $\hat{x}$. Notice that since $(\Phi_{1}(x),\dots \Phi_{d}(x))$ are the ordered roots of the characteristic polynomial of $\hat{x}$, it follows that $\Phi_{i}(x)$ is continuous over $x$ for every $1\leq i\leq d$. Define the sets $\mathcal{H}_{deg}^d$ and $\mathcal{S}_{deg}^d$ by
\begin{align}
\mathcal{H}_{deg}^d
  &:=\{x\in\R^{n_2(d)}\ |\ \Phi_{i}(x)=\Phi_{j}(x),\ \ \text{for some}\ 1\leq i<j\leq d\}\label{eq:Hdeg},\\
\mathcal{S}_{deg}^d
  &:=\{x\in\R^{n_1(d)}\ |\ \Phi_{i}(x)=\Phi_{j}(x),\ \ \text{for some}\ 1\leq i<j\leq d\}\label{eq:Sdeg}.
\end{align}

After identifying the random matrix $Y^{\beta}(t)$  defined in (\ref{ybeta}) as a random vector with values  in $\R^{n_\beta(d)}$, we have that  
$$\{\lambda_{i}^1(t)=\lambda_{j}^1(t)\ \ \text{for some }\ t\in I\ \text{and}\ 1\leq i<j\leq n\}=\{Y^1(t)\in \mathcal{S}_{deg}^d\ \ \text{for some }\ t\in I\},$$
and 
$$\{\lambda_{i}^2(t)=\lambda_{j}^2(t)\ \ \text{for some }\ t\in I\ \text{and}\ 1\leq i<j\leq n\}=\{Y^2(t)\in \mathcal{H}_{deg}^d\ \ \text{for some }\ t\in I\}.$$
Thus, in order to prove Theorem \ref{thm:main}, it suffices to study the hitting probability of $Y^1(t)$ to $\mathcal{S}_{deg}^d$ and $Y^2(t)$ to $\mathcal{H}_{deg}^d$. To this end, we will give a description of some geometric properties of $\mathcal{S}_{deg}^d$ and $\mathcal{H}_{deg}^d$ and determine their Hausdorff dimension. The main results of this section are Propositions \ref{lem:geometricaux1}, \ref{propaux1a},  \ref{lem:geometricaux2} and \ref{lem:geometricaux2GUE} which, roughly speaking, state that there exist measurable sets $\Sc_{in}^{d},\Sc_{out}^{d}\subset\R^{n_1(d)}$ and $\Hc_{in}^{d},\Hc_{out}^{d}\subset\R^{n_2(d)}$, satisfying
\begin{align*}
\Sc_{in}^{d}\subset \Sc_{deg}^d\subset \Sc_{out}^{d}\ \ \ \ \ \ \ \ \ \ \ \ \text{and}\ \ \ \ \ \ \ \ \ \ \ \ 
\Hc_{in}^{d}\subset \Hc_{deg}^d\subset \Hc_{out}^{d},
\end{align*}
as well as the following properties:
\begin{enumerate}
\item $\Sc_{in}^{d}$ and $\Hc_{in}^{d}$ are manifolds of dimensions $n_1(d)-2$ and $n_2(d)-3$, respectively.
\item $\Sc_{out}^{d}$ is  the image of a smooth function defined in an open subset of $\R^{n_1(d)-2}$ with values in $\R^{n_1(d)}$ and $\Hc_{in}^{d}$ is the image of  a smooth function defined in an open subset of $\R^{n_2(d)-3}$ with values in $\R^{n_2(d)}$.
\item  For all $x\in  \Sc_{deg}^d$ satisfying $\textbf{Sp}(x)=d-1$, there exists $\varepsilon>0$ such that $\Sc_{in}^{d}\cap B_{\varepsilon}(x) = \Sc_{deg}^d\cap B_{\varepsilon}(x)$. Similarly,  for all $x\in  \Hc_{deg}^d$ satisfying $\textbf{Sp}(x)=d-1$, there exists $\varepsilon>0$ such that $\Hc_{in}^{d}\cap B_{\varepsilon}(x) = \Hc_{deg}^d\cap B_{\varepsilon}(x)$.

\end{enumerate}

\noindent It is worth mentioning that the sets $\Sc_{deg}^{d}$ and $\Hc_{deg}^d$ can be described as the zeros of a multivariate polynomial (see Appendix A.4 in \cite{AnGuZe} for a proof of this fact) and their dimension is well known in random matrix theory (see for instance \cite{Keller}). However, for the purposes of our application, we are not only interested in the description of $\Sc_{deg}^{d}$ and $\Hc_{deg}^d$ as the set of zeros of a given polynomial, since these type of geometric objects may have corners and self-intersections. Instead, we want to describe a large class of points around which $\Sc_{deg}^{d}$ and $\Hc_{deg}^d$ are differentiable manifolds, as this property is vital for the proof of \eqref{eq:fbmprobone} in Corollary \ref{cor:main}. For this reason, in the rest of the section we provide a self-contained description of the differentiability properties of $\mathcal{S}_{deg}^d$ and $\mathcal{H}_{deg}^d$.

We will require the following terminology from differential geometry. In the sequel, for every $n\in\N$, $x\in\R^{n}$ and $\delta>0$, we will denote by $B_{\delta}(x)$ the open ball of radius $\delta$ and center  $x$. In addition, we will say that an $\R^{n}$-valued function, defined over an open subset of $\R^{m}$ with $m\in\N$, is smooth, if it is infinitely differentiable. 

\begin{Definition}\label{eq:Fdiff}
Let $m,n\in\N$ be such that $m\leq n$. A set $M\subset\R^{n}$ is a smooth submanifold of $\R^{n}$, with dimension $m$, if for every $x_{0}\in M$, there exists $\varepsilon>0$, an open neighborhood of zero $U\subset\R^{m}$ and a smooth mapping 
$$F:U\rightarrow M,$$
satisfying $F(0)=x_{0}$, as well as the following properties:
\begin{enumerate}
\item[-] $F$ is a homeomorphism from $U$ to $M\cap B_{\varepsilon}(x_{0})$. 
\item[-] For every $p\in U$, the derivative of $F$ at $p$, denoted by $DF_p$,  is an injective mapping.
\end{enumerate}
If such mapping $F$ exists, we call it a local chart for $M$ covering $x_{0}$.
\end{Definition}

 If $M$ is a smooth submanifold of $\R^{n}$, we define its tangent plane at a given point $x\in M$, denoted by $TM_{x}$, as the set of vectors of the form  $\alpha^{\prime}(0)$, where $\alpha:(-1,1)\rightarrow M$ is a smooth curve satisfying $\alpha(0)=x$.

 Let $M$ and $N$ be smooth manifolds. We say that $f:M\rightarrow N$ is smooth if for every $x\in M$ and all charts $F$ and $G$, covering $x$ and $f(x)$ respectively, the function $ G^{-1}\circ f\circ F$ is smooth. In this case, we can define the derivative of $f$ at a given point $x\in M$, as the function $Df_{x}:TM_{x}\rightarrow TN_{f(x)}$, that maps every vector $v\in TM_{x}$ of the form $v=\alpha^{\prime}(0)$, to the vector $Df_{x}(v):=\frac{d}{dt}f(\alpha(t))|_{t=0}$.

 Let $f:M\rightarrow N$ be a smooth mapping between manifolds $M,N\subset\R^{n}$. We say that a point $y\in N$ is a regular value for $f$, if for all $x\in f^{-1}\{y\}$, the derivative $Df_{x}:TM_{x}\rightarrow TN_{y}$ is surjective. The following result allows us to identify the level curves of a smooth function, as smooth manifolds. Its proof can be found, for instance, in \cite[Theorem~9.9]{Tu}.

\begin{Theorem}[Preimage theorem]\label{Thm:preimage}
Consider a smooth mapping $f:M\rightarrow N$, where $M$ and $N$ are smooth submanifolds of $\R^{n}$ of dimensions $m_{M}$ and $m_{N}$ respectively, with $m_{N}\leq m_{M}\leq n$. If $y\in N$ is a regular value for $f$, then $f^{-1}\{y\}$ is a smooth submanifold of $\R^{n}$ of dimension $m_{M}-m_{N}$.
\end{Theorem}

Along the paper we will denote by $\|\cdot \|$  the Euclidean norm  on $\R^N$ and by $\langle \cdot, \cdot \rangle$ the corresponding inner product.
We will use the same notation for the norm and inner product in  $\mathbb{C}^N$.

 For $d,h\in\N$, let $\R^{d\times h}$ denote the set of real matrices of dimensions $d\times h$ and let $I_{d}$ be the identity element of $\R^{d\times d}$. For every integer $0\leq i\leq d$, we define the sets 
\begin{align}\label{eq:Otildedef}
\mathcal{O}(d;i):=\{A\in\R^{d\times (d-i)}: A^*A=I_{d-i}\},
\end{align}
where $A^*$ is the transpose of $A$.
In the case where $i=0$, the set $\mathcal{O}(d;i)$ is the orthogonal group of dimension $d$, which will be denoted simply by $\Oc(d):=\mathcal{O}(d;0)$. Using the preimage theorem, we can show that $\Oc(d;i)$ is a submanifold of $\R^{d\times (d-i)}\cong\R^{d(d-i)}$, of dimension $\frac{d(d-1)-i(i-1)}{2}$. This result can be proved in the following manner: consider the mapping $f:\R^{d\times (d-i)}\rightarrow\Sc(d-i),$ defined by 
$$f(X):=X^*X-I_{d-i}.$$ 
Then, for every $A\in f^{-1}\{0\}$, the derivative of $f$ at $A$, denoted by $Df_{A}$, satisfies
\begin{align}\label{eq:DfAB}
Df_{A}B
 &=A^*B+B^*A,\ \ \ \ \ \ \ \text{for every}\ B\in\R^{d\times (d-i)}.
\end{align} 
In particular, for every $C\in\Sc(d-i)$, the matrix $B:=\frac{1}{2}AC$ satisfies $Df_{A}B=C$, so that $Df_{A}$ is surjective for every $A\in f^{-1}\{0\}$. Consequently, zero is a regular value for $f$, and by the preimage theorem, $\Oc(d;i)=f^{-1}\{0\}$ is a smooth submanifold of $\R^{d\times (d-i)}$ of dimension ${\rm dim}(\R^{d(d-i)})-{\rm dim}(\Sc(d-i))=\frac{d(d-1)-i(i-1)}{2}$.

 Similarly, for $d,h\in\N$ we denote by $\C^{d\times h}$ the set of complex matrices of dimensions $d\times h$, and define 
\begin{align}\label{eq:Utildedef}
\mathcal{U}(d;i):=\{A\in\C^{d\times (d-i)} : A^*A=I_{d-i}\},
\end{align}
where $A^*$ denotes the conjugate of the transpose of $A$.
Proceeding as before, we can show that $\mathcal{U}(d;i)$ is a  smooth submanifold of $\C^{d\times(d-i)}\cong\R^{2d(d-i)}$, of dimension $d^2-i^2$. In particular, the unitary group $\mathcal{U}(d):=\mathcal{U}(d;0)$ has dimension $d^2.$

 In the sequel, for every $A\in\C^{d\times h}$, we will denote by $A_{*,j}$ the $j$-th column of $A$, where $1\leq j\leq h$. Next we will show  the following technical result.
 
 \begin{lemma}
For every $R\in\mathcal{U}(d;2)$, there exists $\gamma>0$, such that the set
\begin{align}\label{eq:Vtilde}
\mathcal{V}_{\gamma}^{R}:=\{A\in \mathcal{U}(d;2)\cap B_{\gamma}(R) : \Ip{A_{*,j},R_{*,j}}=|\Ip{A_{*,j},R_{*,j}}|\ \ \ \text{ for }\ 1\leq j\leq d-2\},
\end{align}
is a $(d^2-d-2)$-dimensional submanifold of $\mathcal{U}(d;2)\cap B_{\gamma}(R)$. 
\end{lemma}

\begin{proof} Consider the manifold 
$$\mathbb{T}^{d-2}:=\{(e^{\textbf{i}\theta_1},\dots, e^{\textbf{i}\theta_{d-2}})\in\C^{d-2} :  \theta_i\in[-\pi/2,\pi/2)\}.$$ 
We will prove that if $\gamma>0$ is sufficiently small, the point $\vec{1}:=(1,\dots, 1)$ is a regular value for the smooth function $f:\mathcal{U}(d;2)\cap B_{\gamma}(R)\rightarrow\mathbb{T}^{d-2}$, defined by 
\begin{align}\label{f:def}
f(A):=(|\Ip{A_{*,1},R_{*,1}}|^{-1}\Ip{A_{*,1},R_{*,1}},\dots, |\Ip{A_{*,d-2},R_{*,d-2}}|^{-1}\Ip{A_{*,d-2},R_{*,d-2}}).
\end{align}
Notice that $\mathcal{U}(d;2)$ is a $(d^2-4)$-dimensional manifold. This implies, by Theorem \ref{Thm:preimage},  that the set $\mathcal{V}_{\gamma}^{R}=f^{-1}\{\vec{1}\}$ is a $(d^2-d-2)$-dimensional manifold. To check that $\vec{1}$ is a regular value for $f$, notice that the tangent plane to $\mathbb{T}^{d-2}$ at $\vec{1}$, consists of the the set of vectors $\eta\in\C^{d-2}$ of the form $\eta=(\textbf{i}\eta_{1},\dots, \textbf{i}\eta_{d-2})$, for $\eta_{i}\in\R$. For such $\eta$, there exists $\delta>0$, such that the mapping $A:(-\delta,\delta)\rightarrow \mathcal{V}_{\gamma}^R$, given by  
$$A_{i,j}(t)=e^{\textbf{i}\eta_{j}t}R_{i,j},$$ 
is a curve inside of $\mathcal{U}(d;2)\cap B_{\gamma}(R)$, satisfying 
$Df_{R}(\frac{d}{dt}f(A(t))\big|_{t=0})=\eta$. This proves that $\vec{1}$ is indeed a regular value of $f$.
\end{proof}

The next lemma is a refinement of the well-known continuity property for the eigenprojections of real symmetric matrices. In the sequel, $\mathcal{D}(d)$ will denote the set of diagonal real matrices of dimension $d$. In addition, for every $A\in\C^{d\times d}$, the set $\textbf{Sp}(A)$ will denote the spectrum of $A$ and for $\lambda\in \textbf{Sp}(A)$, $\textbf{E}_{\lambda}^A$ will denote the eigenspace associated to $\lambda$. For every $w^{1},\dots w^{h}\in\C^{d}$, with $h\in\N$, we will denote by $[w^1,\dots, w^{h}]$ the element of $\C^{d\times h}$, whose $j$-th column is equal to $w^{j}$ for all $1\leq j\leq h$.

\begin{Lemma}\label{lem:eigenprojections}
Let $A$ be a $d\times d$ real symmetric matrix, with $|\textbf{Sp}(A)|=d-1$, such that 
\begin{align*}
A
  &=PDP^*,
\end{align*}
for some $P\in\Oc(d)$ and $D\in\mathcal{D}(d)$. Then, for every $\varepsilon>0$ there exists $\delta>0$, such that for all $B\in\Sc_{deg}^d$ satisfying 
\begin{align}\label{eq:ABnear}
\max_{1\leq i,j\leq d}\Abs{A_{i,j}-B_{i,j}}<\delta,
\end{align}
there exists a spectral decomposition of the form $B=Q\Delta Q^*$, where $Q\in\Oc(d)$ and $\Delta\in\mathcal{D}(d)$ satisfy
\begin{align}
\max_{1\leq i,j\leq d}\Abs{Q_{i,j}-P_{i,j}}<\varepsilon\label{eq:QPDdelta1}
\end{align}
and
\begin{align}
\max_{1\leq i\leq d}\Abs{D_{i,i}-\Delta_{i,i}}<\varepsilon\label{eq:QPDdelta2}.
\end{align}
\end{Lemma}
\begin{proof}
The existence of a matrix $\Delta$ satisfying \eqref{eq:QPDdelta2} follows from the continuity of $\Phi$, so it suffices to prove \eqref{eq:QPDdelta1}. The idea for proving this relation is the following: first we express the eigenprojections of the degenerate symmetric matrices lying within a small neighborhood $U$ around $A$, as matrix-valued Cauchy integrals. This representation allows us to prove that the mapping that sends an element $B\in U$, to the eigenprojection of $B$ over its $i$-th largest eigenvalue, is continuous with respect to the entries of $B$. Finally, we will choose a set of eigenvectors for $B$ by applying the (continuous) eigenprojections of $B$ to the eigenvectors of $A$. The matrix $Q$, with columns given by the renormalization of such eigenvectors will then satisfy \eqref{eq:QPDdelta1}.

The detailed proof is as follows. Define $\lambda_{i}:=D_{i,i}$ for $1\leq i\leq d$, and assume without loss of generality that $\lambda_{1}\leq \cdots\leq \lambda_{d-1}=\lambda_{d}$. Using the fact that $|\textbf{Sp}(A)|=d-1$, we get
\begin{align}\label{eq:Dordered}
\lambda_{1}<\lambda_{2}<\cdots<\lambda_{d-2}<\lambda_{d-1}=\lambda_{d}.
\end{align}
  For $i=1,\dots, d$, let $\Cc_{A}^i\subset\C\backslash \textbf{Sp}(A)$ be any smooth closed curve around $\lambda_{i}$ and denote by $\mathcal{I}_{A}^i$ the closure of the interior of $\Cc_{A}^i$. Assume that $\mathcal{C}_{A}^{d-1}=\mathcal{C}_{A}^d$ and that
the diameter of $\mathcal{C}_{A}^{i}$ is sufficiently small, so that $\mathcal{I}_{A}^1,\dots, \mathcal{I}_{A}^{d-1}$ are disjoint. For $\delta>0$, define the set 
$$V_{\delta}:=\{B\in\mathcal{S}_{deg}^d\ |\ \max_{1\leq i,j\leq d}\Abs{A_{i,j}-B_{i,j}}<\delta\}.$$
Using \eqref{eq:Dordered}, as well as the continuity of $\Phi_{1},\dots, \Phi_{d}$ and the fact that $V_{\delta}\subset\mathcal{S}_{deg}^d$, we can easily show that there exists $\delta>0$, such that for all $B\in V_{\delta}$,
\begin{align}\label{eq:PhiBordered}
\Phi_{1}(B)<\Phi_{2}(B)<\cdots<\Phi_{d-2}(B)<\Phi_{d-1}(B)=\Phi_{d}(B),
\end{align}
and 
\begin{align}\label{eq:PhiinIi}
\Phi_{i}(B)\in \mathcal{I}_{A}^i\ \ \ \ \ \text{for all}\ \ B\in V_{\delta}\ \text{ and }\ 1\leq i\leq d.
\end{align}
For such $\delta$, define the mapping $\kappa_{A}^{i}:V_{\delta}\rightarrow\mathcal{S}(d)$, by
\begin{align}\label{eq:holfunccal}
\kappa^{i}_{A}(B)
  &:=\frac{1}{2\pi\textbf{i}}\int_{\Cc_A^i}(\xi I_d-B)^{-1}d\xi.
\end{align}
The matrix $\kappa^{i}_{A}(B)$ is the projection over the sum of the eigenspaces associated to eigenvalues of $B$ inside of $\mathcal{I}_{A}^i$
(see \cite[page 200, Theorem 6]{lax}). Thus, using \eqref{eq:PhiBordered}, \eqref{eq:PhiinIi} and the fact that $\mathcal{I}_{A}^1,\dots,\mathcal{I}_{A}^{d-1}$ are disjoint, we conclude that $\kappa^{i}_{A}(B)$ is the projection over $\textbf{E}_{\Phi_{i}(B)}^B$, for all $1\leq i\leq d$.

 From \eqref{eq:holfunccal}, it follows that the  mapping $B\mapsto \kappa_{A}^{i}(B)$, defined on $ V_{\delta}$, is a continuous function of the entries of $B$. Let $v^{1},\dots, v^{d}$ denote the columns of $P$ and define
\begin{align}\label{eq:wireg}
w^{j}
  &:=\frac{\kappa_{A}^{j}(B)v^{j}}{\|\kappa_{A}^{j}(B)v^{j}\|},
\end{align}
for $1\leq j\leq d-1$ and 
\begin{align}\label{eq:wilast}
w^{d}
  &:=\frac{\kappa_{A}^{d}(B)v^{d}}{\|\kappa_{A}^{d}(B)v^{d}\|}-\frac{\langle\kappa_{A}^{d}(B)v^{d},\kappa_{A}^{d-1}(B)v^{d-1}\rangle}{\|\kappa_{A}^{d}(B)v^{d}\|\|\kappa_{A}^{d-1}(B)v^{d-1}\|^2}\kappa_{A}^{d-1}(B)v^{d-1}.
\end{align}
Since $\kappa^{j}_{A}(B)$ is the projection over $\textbf{E}_{\Phi_{i}(B)}^B$, for all $1\leq j\leq d$ and $B\in V_{\delta}$, we can easily check that $w^{1},\dots, w^{d}$ are orthonormal eigenvectors for $B$. Thus, using the continuity of $\kappa_{A}^{j}$ and the fact that $\kappa_{A}^{i}(A)v^j=v^j$ for all $1\leq j\leq d$, we deduce that there exists $\delta^{\prime}>0$, such that for all $B\in V_{\delta^{\prime}}$, the vectors $w^{1},\dots, w^{d}$ given by \eqref{eq:wireg} and \eqref{eq:wilast}, form an orthonormal base of eigenvectors for $B$ satisfying 
\begin{align*}
\max_{1\leq i,j\leq d}\Abs{v_{i}^{j}-w_{i}^{j}}<\varepsilon,
\end{align*}
where 
\begin{align*}
v^{j}=(v_{1}^j,\dots, v_{d}^j),\ \ \ \ \ \text{and}\ \ \ \ \ w^j=(w_1^j,\dots w_d^j).
\end{align*}
Thus, the matrix $Q=[w^1,\dots,w^d]$ satisfies $B=Q\Delta Q^{*}$ and \eqref{eq:QPDdelta1}, as required.
\end{proof}

The next result is the complex version of Lemma \ref{lem:eigenprojections}, where the sets $\Sc(d)$ and $\Oc(d)$ are replaced by $\Hc(d)$ and $\mathcal{U}(d)$, respectively.
\begin{Lemma}\label{lem:eigenprojectionsu}
Let $A$ be a $d\times d$ complex Hermitian matrix, with $|\textbf{Sp}(A)|=d-1$, such that
\begin{align*}
A
  &=PDP^*,
\end{align*}
for some $P\in\mathcal{U}(d)$ and $D\in\mathcal{D}(d)$. Then, for every $\varepsilon>0$, there exists $\delta>0$ such that for all $B\in\Hc_{deg}^d$ satisfying 
$$\max_{1\leq i,j\leq d}\Abs{A_{i,j}-B_{i,j}}<\delta,$$
there exist a spectral decomposition of the form $B=Q\Delta Q^*$, where $Q\in\mathcal{U}(d)$ and $\Delta\in\mathcal{D}(d)$ satisfy the relations
\begin{align*}
\max_{1\leq i,j\leq d}\Abs{Q_{i,j}-P_{i,j}}<\varepsilon\ \ \ \ \ \ \text{ and }\ \ \ \ \ \ \max_{1\leq i\leq d}\Abs{D_{i,i}-\Delta_{i,i}}<\varepsilon.
\end{align*}
\end{Lemma}
\begin{proof}
It follows from arguments similar to those used in the proof of Lemma \ref{lem:eigenprojections}.
\end{proof}

Define the function $\Lambda:\R^{d-1}\rightarrow\mathcal{D}(d)$, that maps the vector $\beta=(\beta_{1},\dots, \beta_{d-1})\in\R^{d-1}$, to the matrix $\Lambda(\beta)=\{\Lambda_{i,j}(\beta);1\leq i,j\leq d\}$, given by   
\begin{align}\label{eq:Lambdadef}
\Lambda_{i,j}(\beta)
  &:=\left\{\begin{array}{lll}\delta_{i,j}\beta_{i} &\text{ if }  &1\leq i\leq d-2\\
	\delta_{i,j}\beta_{d-1}                             &\text{ if }  &i=d-1,d.\end{array}\right.
\end{align}
In the next  proposition, we bound from above the set $\Sc_{deg}^d$.
\begin{Proposition}\label{lem:geometricaux1}
There exists a compactly supported smooth function $\Pi:\R^{\frac{d(d-1)}{2}-1}\rightarrow\R^{d\times d}$, such that the mapping $F:\R^{\frac{d(d-1)}{2}-1}\times\R^{d-1} \rightarrow  \mathcal{S}(d)$, defined by
\begin{align}\label{eq:Fdef}
F(\alpha,\beta):=\Pi(\alpha)\Lambda(\beta)\Pi(\alpha)^*,
\end{align}
for $\alpha\in \R^{\frac{d(d-1)}{2}-1}$ and $\beta\in\R^{d-1}$, satisfies 
\begin{align}\label{inclusion:sdeggoe}
\Sc_{deg}^d\subset\mathcal{S}_{out}^d:=\{x\in\R^{n_1(d)} : \hat{x}\in  {\rm Im}(F)\}.
\end{align}
\end{Proposition}
\begin{proof}
For $\varepsilon>0$, define the interval $J_{\varepsilon}:=(-\varepsilon,\varepsilon)^{\frac{d(d-1)}{2}-1}$. First we reduce the problem, to proving that there exist $L\in\N$ and smooth functions $\Pi^{1},\dots, \Pi^{L}:\R^{\frac{d(d-1)}{2}-1}\rightarrow \R^{d\times d}$,  supported in $J_{\varepsilon}$, such that the mappings 
$F^l:J_{\varepsilon}\times\R^{d-1}  \rightarrow  \mathcal{S}(d)$, defined by
\begin{align}\label{eq:Fidef}
F^{l}(\alpha,\beta):=\Pi^{l}(\alpha)\Lambda(\beta)\Pi^{l}(\alpha)^*,
\end{align}
for $1\leq l\leq L$, $\alpha\in J_{\varepsilon}$ and $\beta\in \R^{d-1}$, satisfy 
\begin{align}\label{eq:Scdegfmany}
\Sc_{deg}^d
  &\subset\{x\in\R^{n_1(d)} : \hat{x}\in \bigcup_{l=1}^{L}{\rm Im}(F^{l})\}.
\end{align}
 To show this reduction, notice that if \eqref{eq:Scdegfmany} holds, then any smooth function $\Pi$, supported in $J_{3\varepsilon L}$, satisfying
\begin{align*}
\Pi(x)
  &:=\Pi^{l}(x-3l\varepsilon,0,\dots,0))\ \ \ \ \ \ \text{if}\ \ \ \ \ \ x\in B_{\varepsilon}(3l\varepsilon,0,\dots,0)\subset\R^{\frac{d(d-1)}{2}-1},
\end{align*}
is such that the mapping \eqref{eq:Fdef} satisfies \eqref{inclusion:sdeggoe}.

  Therefore, it suffices to find $\Pi^{1},\dots, \Pi^{L}$. The heuristics for constructing such functions is the following: every matrix $X\in\Sc_{deg}^d$ can be expressed in the form 
$$X=PDP^*,$$
with $D\in\mathcal{D}(d)$ and $P\in\Oc(d)$. Since $X$ is degenerate, we have some flexibility for choosing $P$, due to the fact that if $X$ has eigenvalues $\mu_{1},\dots, \mu_{d}$, and $\mu_{h}=\mu_{h+1}$, then the eigenspaces $\textbf{E}_{\mu_{j}}^{X}$, with $\mu_j\neq \mu_h$, completely determine $\textbf{E}_{\mu_{h}}^{X}$. This allows us to construct $P$ by describing only the eigenvectors associated to $\textbf{E}_{\mu_{j}}^{X}$, with $\mu_j\neq \mu_h$. We can show that these spaces can be locally embedded into the set $\mathcal{O}(d;2)$, which has dimension $\frac{d(d-1)}{2}-1$. Then we extend such local embeddings to compactly supported $\R^{d\times d}$-valued functions, and apply a compactness argument to obtain the existence of $\Pi^{1},\dots, \Pi^{L}$.

 The detailed construction is as follows. For each matrix $R\in\mathcal{O}(d;2)$, we have that $R^*R=I_{d-2}$, and thus, the columns of $R$ are orthonormal. As a consequence, by completing $\{R_{*,1},\dots, R_{*,d-2}\}$ to an orthonormal basis of $\R^{d}$, we can choose an element $P\in \mathcal{O}(d)$,  such that $P_{*,j}=R_{*,j}$ for all $1\leq j\leq d-2$.  Since $\mathcal{O}(d;2)$ is a smooth manifold of dimension $\frac{d(d-1)}{2}-1$, we have that if $\gamma>0$ is sufficiently small, the set $\mathcal{O}(d;2)\cap B_{\gamma}(R)$ can be parametrized with a chart $\varphi$, defined on $J_{\varepsilon}$, for some $\varepsilon>0$. Namely, the mapping 
$$\varphi:J_{\varepsilon}\rightarrow \mathcal{O}(d;2)\cap B_{\gamma}(R)$$
is a diffeomorphism satisfying $\varphi(0)=R$. Denote by $\varphi_{*,j}$ the $j$-th column vector of $\varphi$. By construction, every matrix $S\in\Oc(d;2)$ of the form $S=\varphi(\alpha)$, with $\alpha\in J_{\varepsilon}$, satisfies $\Norm{P_{*,j}-S_{*,j}}<\gamma$ for all $1\leq j\leq d-2$, and thus, for $\gamma$ sufficiently small,
\begin{multline*}
\Abs{\|P_{*,d-1}-\sum_{j=1}^{d-2}\Ip{S_{*,j},P_{*,d-1}}S_{*,j}\|-1}\\
\begin{aligned}
  &=\Abs{\|P_{*,d-1}-\sum_{j=1}^{d-2}\Ip{S_{*,j},P_{*,d-1}}S_{*,j}\|-\|P_{*,d-1}-\sum_{j=1}^{d-2}\Ip{P_{*,j},P_{*,d-1}}P_{*,j}\|}<\frac{1}{2}.
\end{aligned}
\end{multline*} 
As a consequence, $\big\|P_{*,d-1}-\sum_{j=1}^{d-2}\Ip{\varphi_{*,j}(\alpha),P_{*,j}}\varphi_{*,j}(\alpha)\big\|$ is bounded away from zero for all $\alpha\in J_{\varepsilon}$, and hence, the mapping $\alpha\mapsto\psi_{1}(\alpha)$, with 
\begin{align}\label{eq:psi1def}
\psi_{1}(\alpha)
  &:=\frac{P_{*,d-1}-\sum_{j=1}^{d-2}\Ip{\varphi_{*,j}(\alpha),P_{*,d-1}}\varphi_{*,j}(\alpha)}{\big\|P_{*,d-1}-\sum_{j=1}^{d-2}\Ip{\varphi_{*,j}(\alpha),P_{*,d-1}}\varphi_{*,j}(\alpha)\big\|}
\end{align}
is smooth. Proceeding similarly, we can show that for $\gamma$ sufficiently small, the mapping $\alpha\mapsto\psi_{2}(\alpha)$, with
\begin{align}\label{eq:psi2def}
\psi_{2}(\alpha)
  &:=\frac{P_{*,d}-\Ip{\psi_{1}(\alpha),P_{*,d}}\psi_{1}(\alpha)-\sum_{j=1}^{d-2}\Ip{\varphi_{*,j}(\alpha),P_{*,d}}\varphi_{*,j}(\alpha)}{\big\|P_{*,d}-\Ip{\psi_{1}(\alpha),P_{*,d}}\psi_{1}(\alpha)-\sum_{j=1}^{d-2}\Ip{\varphi_{*,j}(\alpha),P_{*,d}}\varphi_{*,j}(\alpha)\big\|}
\end{align}
is smooth as well. Let $\Pi:\R^{\frac{d(d-1)}{2}-1}\rightarrow\R^{d\times d}$ be any  smooth function, supported in $J_{\varepsilon}$, such that for all $\alpha\in J_{\varepsilon/2}$, 
\begin{align}\label{eq:Pidef}
\Pi_{*,j}(\alpha)
  &:=\left\{\begin{array}{ll}\varphi_{*,j}(\alpha)& \text{ if } 1\leq j\leq d-2\\\psi_{1}(\alpha)& \text{ if } j=d-1\\\psi_{2}(\alpha)& \text{ if } j=d.\end{array}\right.
\end{align}
By construction, $\Pi$ has the property that 
\begin{align}
V_{\Pi}^R
  &:=\{[\Pi_{*,1}(\alpha),\dots, \Pi_{*,d-2}(\alpha)] : \alpha\in J_{\varepsilon/2}\}=\varphi(J_{\varepsilon/2}),\label{eq:Vdef}
\end{align}
is an open subset of $\mathcal{O}(d;2)$ containing $R$. Therefore, since $\mathcal{O}(d;2)$ is compact and the collection of sets $\{V_{\Pi}^R : R\in\mathcal{O}(d;2)\}$ is an open cover for $\mathcal{O}(d;2)$, we deduce that there exist $L\in\N$ and smooth $\R^{d\times d}$-valued functions $\Pi^{1},\dots, \Pi^{L}$ of the form \eqref{eq:Pidef}, supported in intervals of the form $J_{\varepsilon_l}$, with $\varepsilon_l>0$, such that the sets 
\begin{align*}
V_{l}
  &=\{[\Pi_{*,1}^l(\alpha),\dots, \Pi_{*,d-2}^l(\alpha)] : \alpha\in J_{\varepsilon/2}\},
\end{align*}
satisfy
\begin{align}\label{eq:Vinclusion}
\mathcal{O}(d;2)
  &= V_{1}\cup\cdots\cup V_{L}.
\end{align}
In the sequel, we will assume without loss of the generality that there exists $\varepsilon>0$, such that $\varepsilon_{l}=\varepsilon$ for all $l=1,\dots, L$.

 By construction, the functions $\Pi^{1},\dots, \Pi^{L}$ are smooth and compactly supported, so it suffices to show that
$$\mathcal{S}_{deg}^d\subset\bigcup_{1\leq l\leq L}\{x\in\R^{n_1(d)} : \hat{x}\in {\rm Im}(F^{l})\},$$
where $F^{1},\dots F^{L}$ are defined by \eqref{eq:Fidef}. To this end, take $x\in\mathcal{S}_{deg}$ and let $Q\in\mathcal{O}(d)$ and $\Delta\in\mathcal{D}(d)$ be such that $\hat{x}=Q\Delta Q^*$. By permuting the diagonal of $\Delta$ and the columns of $Q$ if necessary, we can assume that $\Delta_{d-1,d-1}=\Delta_{d,d}$. Applying \eqref{eq:Vinclusion} to $[Q_{*,1},\dots, Q_{*,d-2}]\in\mathcal{O}(d;2)$, we deduce that there exist $1\leq l\leq L$ and  $\alpha\in J_{\varepsilon}$, such that $[Q_{*,1},\dots, Q_{*,d-2}]=[\Pi_{*,1}^{l}(\alpha),\dots,\Pi_{*,d-2}^{l}(\alpha)]$.

  Let $\Delta=\Lambda(\beta)$ for $\beta\in\R^{d-1}$. To finish the proof, it suffices to show that $\hat{x}=\Pi^{l}(\alpha)\Lambda(\beta)\Pi^{l}(\alpha)^*$. By construction, 
$$\{\Pi_{*,1}^{l}(\alpha),\dots,\Pi_{*,d}^{l}(\alpha)\}\ \ \ \ \ \ \ \text{and}\ \ \ \ \ \ \ \{Q_{*,1},\dots, Q_{*,d}\}$$ 
are orthonormal bases of $\R^{d}$ satisfying 
$$\{\Pi_{*,1}^{l}(\alpha),\dots, \Pi_{*,d-2}^{l}(\alpha)\}=\{Q_{*,1},\dots, Q_{*,d-2}\}.$$ 
Thus,  $\text{span}\{\Pi_{*,d-1}^{l}(\alpha),\Pi_{*,d}^{^{l}}(\alpha)\}=\text{span}\{Q_{*,d-1},Q_{*,d}\}$. In particular, $\text{span}\{\Pi_{*,d-1}^{l}(\alpha),\Pi_{*,d}^{l}(\alpha)\}$ is contained in the eigenspace associated to $\Delta_{d-1,d-1}$, which implies that $\Pi_{*,d-1}^{l}(\alpha),\Pi_{*,d}^{l}(\alpha)$ are orthonormal eigenvectors of $\hat{x}$ with eigenvalue $\Delta_{d-1,d-1}$. From here we conclude that $\{\Pi_{*,1}^{l}(\alpha),\dots,\Pi_{*,d}^{l}(\alpha)\}$ is a basis of eigenvectors for $\hat{x}$, hence implying that 
\begin{align*}
\hat{x}=\Pi^{l}(\alpha)\Lambda(\beta)\Pi^{l}(\alpha)^*,
\end{align*}
as required.
\end{proof}

 In the next proposition, we bound from above the set $\Hc_{deg}^d$.

\begin{Proposition}  \label{propaux1a}
There exists a compactly supported  smooth function $\widetilde{\Pi}:\R^{d^2-d-2}\rightarrow\C^{d\times d}$, such that the mapping $\widetilde{F}:\R^{d^2-d-2}\rightarrow  \Hc(d)$, defined by
\begin{align}\label{eq:Fuildedef}
\widetilde{F}(\alpha,\beta):=\widetilde{\Pi}(\alpha)\Lambda(\beta)\widetilde{\Pi}(\alpha)^*,
\end{align}
for $\alpha\in \R^{d^2-d-2}$ and $\beta\in\R^{d-1}$, satisfies 
\begin{align}\label{inclusion:sdeggue}
\Hc_{deg}^d\subset\{x\in\R^{n_2(d)} : \hat{x}\in{\rm Im}(\widetilde{F})\}.
\end{align}
\end{Proposition}
\begin{proof}
For $\varepsilon>0$, set $\widetilde{J}_{\varepsilon}:=(-\varepsilon,\varepsilon)^{d^2-d-2}$. Similarly to the proof of Proposition \ref{lem:geometricaux1}, it suffices to show that there exist $M\in\N$ and smooth $\C^{d\times d}$-valued functions $\widetilde{\Pi}^{l}$, with $1\leq l\leq M$, supported in $\widetilde{J}_{\varepsilon}$, with $\varepsilon>0$, such that the mappings $\widetilde{F}^{l}:\widetilde{J}_{\varepsilon}\times\R^{d-1} \rightarrow  \mathcal{H}(d)$,  defined by
\begin{align}\label{eq:Ftildeidef}
\begin{array}{lll}
\widetilde{F}^{l}(\alpha,\beta)
  &:=\widetilde{\Pi}^{l}(\alpha)\Lambda(\beta)\widetilde{\Pi}^{l}(\alpha)^*,
\end{array}
\end{align}
satisfy 
\begin{align}\label{eq:Hcdegfmany}
\Hc_{deg}^d
  \subset&\Hc_{out}^d:=\{x\in\R^{n_2(d)} : \hat{x}\in \bigcup_{l=1}^{M}{\rm Im}(\widetilde{F}^{l})\}.
\end{align}

\noindent For each $R\in\mathcal{U}(d;2)$, choose a unitary matrix $P\in\mathcal{U}(d)$, such that $P_{i,j}=R_{i,j}$ for all $1\leq i\leq d$ and $1\leq j\leq d-2$. Using the fact that the set $\mathcal{V}_{\nu}^{R}$, defined by \eqref{eq:Vtilde}, is a smooth manifold of dimension $d^2-d-2$ for $\nu$ sufficiently small, it follows that there exist $\varepsilon,\gamma>0$, and a smooth diffeomorphism $\widetilde{\varphi}:\widetilde{J}_{\varepsilon}\rightarrow \mathcal{V}_{\gamma}^{R}$, such that $\widetilde{\varphi}(0)=R$. Moreover, as in the proof of Proposition \ref{lem:geometricaux1}, if $\gamma$ is sufficiently small, the mappings $\widetilde{\psi}_{1}$ and $\widetilde{\psi}_{2}$ defined as in \eqref{eq:psi1def} and \eqref{eq:psi2def} (when $\varphi$ is replaced by $\widetilde{\varphi}$),
are smooth. Let $\widetilde{\Pi}:\R^{d^2-d-2}\rightarrow\C^{d\times d}$ be any smooth function, supported in $\widetilde{J}_{\varepsilon}^d$, such that for all $\alpha\in \widetilde{J}_{\varepsilon/2}$,
\begin{align}\label{eq:Pitildedef}
\widetilde{\Pi}_{*,j}(\alpha)
  &:=\left\{\begin{array}{ll}\widetilde{\varphi}_{*,j}(\alpha)& \text{ if } 1\leq j\leq d-2\\\widetilde{\psi}_{1}(\alpha)& \text{ if } j=d-1\\\widetilde{\psi}_{2}(\alpha)& \text{ if } j=d.\end{array}\right.
\end{align}
Define the function $\zeta^R:\mathcal{U}(d;2)\cap B_{\gamma}(R)\rightarrow \mathcal{U}(d;2)$ by $\zeta^R(A)=\{\zeta_{i,j}^R(A) ;  1\leq i\leq d\ \ \text{ and }\ \ 1\leq j\leq d-2\}$, where
$$\zeta_{*,j}^R(A):=\Ip{A_{*,j},R_{*,j}}^{-1}|\Ip{A_{*,j},R_{*,j}}|A_{*,j},$$
and the set
\begin{align*}
V_{\widetilde{\Pi},\delta}^R:=\{[\widetilde{\Pi}_{*,1}(\alpha),\dots, \widetilde{\Pi}_{*,d-2}(\alpha)] :  \alpha\in \widetilde{J}_{\delta}\}=\widetilde{\varphi}(\widetilde{J}_{\delta}),
\end{align*}
for $0<\delta<\varepsilon$. By the continuity of the inner product in $\C^{d}$, there exists $0<\varepsilon^{\prime}<\varepsilon/2$, such that 
$$\zeta^R(\widetilde{\varphi}(\widetilde{J}_{\varepsilon^{\prime}}))\subset \widetilde{\varphi}(\widetilde{J}_{\varepsilon}).$$
By construction, $\widetilde{\Pi}(0)=P$ and $V_{\widetilde{\Pi},\varepsilon^{\prime}}^R$ is an open subset of $\mathcal{U}(d;2)$ containing $R$, such that 
\begin{align*}
\zeta^R(V_{\widetilde{\Pi},\varepsilon^{\prime}}^R)
  &\subset V_{\widetilde{\Pi},\varepsilon}^R.
\end{align*}
Therefore, since $\mathcal{U}(d;2)$ is compact and the collection $\{V_{\widetilde{\Pi},\varepsilon^{\prime}}^R : R\in\mathcal{U}(d;2)\}$ is an open cover for $\mathcal{U}(d;2)$, we deduce that there exist $M\in\N$, $\varepsilon_{1}^{\prime},\varepsilon_1,\dots, \varepsilon_{M}^{\prime},\varepsilon_{M}>0$ and smooth $\C^{d\times d}$-valued functions $\widetilde{\Pi}^{1},\dots, \widetilde{\Pi}^{M}$, supported in intervals of the form $\widetilde{J}_{\varepsilon_l}$, with $\varepsilon_{l}^{\prime}<\varepsilon_l/2$, such that the sets 
\begin{align*}
\widetilde{V}_{l}
  &:=\{[\widetilde{\Pi}_{*,1}^l(\alpha),\dots, \widetilde{\Pi}_{*,d-2}^l(\alpha)] : \alpha\in \widetilde{J}_{\varepsilon_l/2}\},
\end{align*}
satisfy
\begin{align}\label{eq:Vtildeinclusion}
\mathcal{U}(d;2)
  &= \widetilde{V}_{1}\cup\cdots\cup \widetilde{V}_{M},
\end{align}
and the matrices $R_l:=[\widetilde{\Pi}_{*,1}^{l}(0),\dots, \widetilde{\Pi}_{*,d-2}^{l}(0)]$, with $1\leq l\leq M$, satisfy 
\begin{align}\label{eq:xiPlinclusion}
\zeta^{R_l}(V_{\widetilde{\Pi}^l,\varepsilon_l^{\prime}}^R)
  &\subset V_{\widetilde{\Pi}^l,\frac{\varepsilon_l}{2}}^R.
\end{align}
In the sequel, we will assume without loss of the generality that there exist $\varepsilon,\varepsilon^{\prime}>0$, such that $\varepsilon_{l}=\varepsilon$ and $\varepsilon_l^{\prime}=\varepsilon^{\prime}$ for all $l=1,\dots, M$.

 By construction, the functions $\widetilde{\Pi}^{1},\dots, \widetilde{\Pi}^{M}$ are smooth and supported in $\widetilde{J}_{\varepsilon}$, so it suffices to show relation \eqref{eq:Hcdegfmany}.  To this end, take $x\in\mathcal{H}_{deg}^d$ and let $\Delta\in\mathcal{D}(d)$, $Q\in\mathcal{U}(d)$ be such that
\begin{align}\label{eq:Bdecomp}
\hat{x}
  &=Q\Delta Q^*.
\end{align}
As in the proof of Proposition \ref{lem:geometricaux1}, we can assume  that $\Delta_{d-1,d-1}=\Delta_{d,d}$ and thus there exists $\beta\in \R^{d-1}$ such that $\Delta=\Lambda(\beta)$. Let $B\in\C^{d\times(d-2)}$ be given by $B_{i,j}=Q_{i,j}$, for $1\leq i\leq d$ and $1\leq j\leq d-2$. By \eqref{eq:Vtildeinclusion}, there exists $1\leq l_{0}\leq M$, such that $B\in \widetilde{\Pi}^{l_{0}}(\widetilde{J}_{\varepsilon^{\prime}})$. Define $P:=\widetilde{\Pi}^{l_{0}}(0)$ and $R\in\C^{d\times(d-2)}$ by $R_{i,j}:=P_{i,j}$ for all $1\leq i\leq d$ and $1\leq j\leq d-2$. Notice that the decomposition \eqref{eq:Bdecomp} still holds if the columns of $Q$ are multiplied by any complex number of unit length. Moreover, by \eqref{eq:xiPlinclusion}, $\zeta^{R}(B)$ belongs to $V_{\widetilde{\Pi}^{l_{0}},\frac{\varepsilon_{l_{0}}}{2}}^R$, and thus, since the columns of $[Q_{*,1},\dots, Q_{*,d-2}]$ are scalar multiples of $\zeta^R(B)$, by replacing the first $d-2$ columns of $Q$ by those of the matrix $\zeta^{R}(B)$ in relation \eqref{eq:Bdecomp}, we can assume that
$$[Q_{*,1},\dots, Q_{*,d-2}]=[\widetilde{\Pi}_{*,1}^{l_{0}}(\alpha),\dots, \widetilde{\Pi}_{*,d-2}^{l_{0}}(\alpha)],$$
for some $\alpha\in \widetilde{J}_{\varepsilon/2}$. To finish the proof, it suffices to show that $\hat{x}=\widetilde{\Pi}^{l_{0}}(\alpha)\Lambda(\beta)\widetilde{\Pi}^{l_{0}}(\alpha)^*$. By construction, 
$$\{Q_{*,1}=\widetilde{\Pi}_{*,1}^{l_{0}}(\alpha),\dots, Q_{*,d-2}=\widetilde{\Pi}_{*,d-2}^{l_{0}}(\alpha),\widetilde{\Pi}_{*,d-1}^{l_{0}}(\alpha),\widetilde{\Pi}_{*,d}^{l_{0}}(\alpha)\}$$ 
and 
$$\{Q_{*,1},\dots, Q_{*,d}\}$$ 
are orthonormal basis of $\C^{d}$, and thus, $\text{span}\{\widetilde{\Pi}_{*,d-1}^{l_{0}}(\alpha),\widetilde{\Pi}_{*,d}^{l_{0}}(\alpha)\}=\text{span}\{Q_{*,d-1},Q_{*,d}\}$. In particular, $\text{span}\{\widetilde{\Pi}_{*,d-1}^{l_{0}}(\alpha),\widetilde{\Pi}_{*,d}^{l_{0}}(\alpha)\}$ is contained in the eigenspace associated to $\Delta_{d-1,d-1}=\Delta_{d,d}$, which implies that $\widetilde{\Pi}_{*,d-1}^{l_{0}}(\alpha),\widetilde{\Pi}_{*,d}^{l_{0}}(\alpha)$ are orthonormal eigenvectors of $\hat{x}$ with eigenvalue $\Lambda_{d-1,d-1}(\beta)$. From here we conclude that 
$$\{\widetilde{\Pi}_{*,1}^{l_{0}}(\alpha),\dots, \widetilde{\Pi}_{*,d}^{l_{0}}(\alpha)\},$$
forms a base of eigenvectors for $\hat{x}$, hence implying that 
\begin{align*}
\hat{x}=\widetilde{\Pi}(\alpha)\Lambda(\beta)\widetilde{\Pi}(\alpha)^*,
\end{align*}
as required. The proof is now complete.
\end{proof}
%
 The following result gives sufficient conditions for points $x_{0}\in\Sc_{deg}$ to have a neighborhood diffeomorphic to $\R^{n_1(d)-2}$.
\begin{Proposition}\label{lem:geometricaux2}
Let $x_0\in \mathcal{S}_{deg}^d$ be such that $|\textbf{Sp}(\hat{x}_0)|=d-1$. Then there exists $\gamma>0$ such that $\Sc_{deg}^d\cap B_{\gamma}(x_0)$ is an $(n_1(d)-2)$-dimensional manifold.
\end{Proposition}

\begin{proof} 
The ideas of the proof are similar to those used in Proposition \ref{lem:geometricaux1}, but in this case, the compactness argument that leads to \eqref{eq:Vinclusion}, is replaced by a localization argument for the matrix of eigenvectors of $\hat{x}_{0}$.

 Let $P\in\mathcal{O}(d)$ and $D\in\mathcal{D}(d)$ be such that 
\begin{align*}
\hat{x}_{0}
  &=PDP^*.
\end{align*}
Since $|\textbf{Sp}(\hat{x}_0)|=d-1$, only one of the eigenvalues $D_{1,1},\dots, D_{d,d}$ of $\hat{x}_{0}$ is repeated. We will assume without loss of generality that $D_{d-1,d-1}=D_{d,d}$. Define $J_{\varepsilon}$, for $\varepsilon>0$, by $J_{\varepsilon}:=(-\varepsilon,\varepsilon)^{\frac{d(d-1)}{2}-1}$, and let $R\in \mathcal{O}(d;2)$ be the matrix $R=\{R_{i,j} ; 1\leq i\leq d,\ \  1\leq j\leq d-2$\}, with $R_{i,j}=P_{i,j}$ for all $1\leq i\leq d$ and $1\leq j\leq d-2$. Since $\mathcal{O}(d;2)$ is a manifold of dimension $\frac{ d(d-1)}2-1$, there exists $\gamma>0$ and a smooth diffeomorphism 
\begin{align*}
\varphi:J_{\varepsilon}\rightarrow \mathcal{O}(d;2)\cap B_{\gamma}(R),
\end{align*}
with $\varphi(0)=R$. Denote by $\varphi_{*,j}$ the $j$-th column vector of $\varphi$. Proceeding as in the proof of Proposition \ref{lem:geometricaux1}, we can show that if $\gamma$ is sufficiently small,
the functions $\psi_1$ and $\psi_2$ defined  in \eqref{eq:psi1def} and \eqref{eq:psi2def}
are smooth. Define $\Pi:J_{\varepsilon}\rightarrow\Oc(d)$ by
\begin{align*}
\Pi_{*,j}(\alpha)
  &:=\left\{\begin{array}{ll}\varphi_{*,j}(\alpha)& \text{ if } 1\leq j\leq d-2\\\psi_{1}(\alpha)& \text{ if } j=d-1\\\psi_{2}(\alpha)& \text{ if } j=d,\end{array}\right.
\end{align*}
and $F:J_{\varepsilon}\times\R^{d-1}\rightarrow\Sc_{deg}^d$ by
\begin{align*}
F(\alpha,\beta):=\Pi(\alpha)\Lambda(\beta)\Pi(\alpha)^*.
\end{align*}

In order to show that $\Sc_{deg}^d\cap B_{\gamma}(x_0)$ is an $(n_1(d)-2)$-dimensional manifold, we will prove that there exist open subsets $U\subset J_{\varepsilon}$ and $V\subset \mathcal{S}_{deg}^d\cap B_{\gamma}(\hat{x}_0)$, such that the mapping 
\begin{align}\label{eq:diffeomorphism}
\begin{array}{lll}U\times\R^{d-1}&\rightarrow&V\\ (\alpha,\beta)&\longmapsto&F(\alpha,\beta)\end{array}
\end{align}
is a diffeomorphism. To this end, define
\begin{align}\label{eq:rdef}
r:=\frac{1}{2}\min_{\substack{\mu,\nu\in\textbf{Sp}(\hat{x}_0)\\\mu\neq\nu}}|\mu-\nu|.
\end{align}
Notice that by Lemma \ref{lem:eigenprojections}, there exists $\delta>0$ satisfying that for all $x\in\mathcal{S}_{deg}^d\cap B_{\delta}(\hat{x}_0)$, there exist $Q\in\mathcal{O}(d)$ and $\Delta\in\mathcal{D}(d)$, such that $\hat{x}=Q\Delta Q^*$, 
\begin{align}\label{eq:QclosetoP}
Q\in \mathcal{O}(d)\cap B_{\gamma/2}(P),
\end{align}
and 
\begin{align}\label{eq:Deltaradiusr}
\Delta\in \mathcal{D}(d)\cap B_{r}(D).
\end{align}
By \eqref{eq:QclosetoP}, there exists $\alpha\in J_{\varepsilon}$ such that $\varphi(\alpha)=[Q_{*,1},\dots, Q_{*,d-2}]$. As a consequence, since 
$$\{\Pi_{*,1}(\alpha),\dots,\Pi_{*,d}(\alpha)\}\ \ \ \ \ \ \ \text{and}\ \ \ \ \ \ \ \{Q_{*,1},\dots, Q_{*,d}\}$$ 
are orthonormal bases of $\R^{d}$ satisfying 
$$\{\Pi_{*,1}(\alpha),\dots, \Pi_{*,d-2}(\alpha)\}=\{Q_{*,1},\dots, Q_{*,d-2}\},$$ 
we have that $\text{span}\{\Pi_{*,d-1}(\alpha),\Pi_{*,d}(\alpha)\}=\text{span}\{Q_{*,d-1},Q_{*,d}\}$. On the other hand, by \eqref{eq:Deltaradiusr}, we have that $\Delta_{1,1}<\cdots<\Delta_{d-1,d-1}=\Delta_{d,d}$, and thus,  we conclude that $\Pi_{*,d-1}(\alpha),\Pi_{*,d}(\alpha)$ are eigenvectors of $\hat{x}$ with eigenvalue $\Delta_{d-1,d-1}$, hence implying that  
$$\{\Pi_{*,1}(\alpha),\dots,\Pi_{*,d}(\alpha)\}$$
is a basis of eigenvectors for $\hat{x}$ and
\begin{align*}
\hat{x}=\Pi(\alpha)\Lambda(\beta)\Pi(\alpha)^*.
\end{align*}
From here it follows that if $U\subset\R^{n_{1}(d)-2}$ and $V\subset\mathcal{S}_{deg}^d$ are given by  $V:=B_{\delta}(\hat{x}_0)$ and $U:=F^{-1}(V)$, the mapping \eqref{eq:diffeomorphism} is onto. Therefore,  in order to show that  the mapping $F$ defined in \eqref{eq:diffeomorphism} is a diffeomorphism, it suffices to show that the following conditions hold:   
\begin{enumerate}
\item[(i)]     The restriction of $F$ to $U$ is injective,
\item[(ii)]    The function $F^{-1}$ is continuous over $V$, 
\item[(iii)]   $D_{p}F$ is injective for every $p\in J_{\varepsilon}\times\R^{d-1}$.
\end{enumerate}
Notice that condition (iii) implies that $F$ is locally injective, which gives condition (i) for $\delta>0$ sufficiently small. Hence, it suffices to show that $F^{-1}$ is continuous and $D_{p}F$ is injective for every $p\in J_{\varepsilon}\times\R^{d-1}$. We split the proof of these claims into the following  two steps:

\medskip
\noindent {\it Step 1}.
First we show that $F^{-1}$ is continuous. Consider a sequence $\{y_{n}\}_{n\geq 1}\subset \Sc_{deg}^d\cap B_{\delta}(\hat{x}_{0})$ such that  $\lim_{n}y_{n}=y$ for some $y\in \Sc_{deg}^d\cap B_{\delta}(\hat{x}_{0})$. Consider the elements $(\alpha_{n},\beta_n),(\alpha,\beta)\in J_{\varepsilon}\times \R^{d-1}$, defined by $(\alpha_{n},\beta_n)=F^{-1}(y_{n})$ and  $(\alpha,\beta):=F^{-1}(y)$,  that satisfy
\[
 y_{n}
  =\Pi(\alpha_n)\Lambda(\beta_{n})\Pi(\alpha_n) \]
  and
  \begin{equation}\label{eq:piFalphabeta}
   y=\Pi(\alpha)\Lambda(\beta)\Pi(\alpha).
  \end{equation}
Our aim is to show that $\lim_{n}\alpha_{n}=\alpha$ and $\lim_{n}\beta_{n}=\beta$. Condition $\lim_{n}\beta_{n}=\beta$ follows from the continuity of $\Phi_1,\dots, \Phi_{d}$. To show that 
\begin{align}\label{eq:limalphasubseq}
\lim_{n}\alpha_n=\alpha,
\end{align}
we proceed as follows. By construction, for all $n\in\N$, $\Pi(\alpha_{n})\in \mathcal{O}(d)\cap B_{\gamma/2}(P)$, and thus $\varphi(\alpha_n)\in \mathcal{O}(d;2)\cap B_{\gamma/2}(R)$. As a consequence, the sequence $\{\alpha_n\}_{ n\geq 1}$ is contained in the compact set $K:=\varphi^{-1}(\mathcal{O}(d;2)\cap \overline{B_{\gamma/2}(R)})$. Therefore, it suffices to show that every convergent subsequence  $\{\alpha_{m_n}\}_{n\geq 1}\subset\{\alpha_{n}\}_{n\geq 1}$, satisfies $\lim_{n}\alpha_{m_n}=\alpha$.

 By taking limit as $n\rightarrow\infty$ in the relation $y_{m_{n}}=\Pi(\alpha_{m_n})\Lambda(\beta_{m_n})\Pi(\alpha_{m_n})^*$, we get
\begin{align}\label{eq:yFalphabeta}
y
  &=\Pi(\lim_{n}\alpha_{m_n})\Lambda(\beta)\Pi(\lim_{n}\alpha_{m_n})^*.
\end{align}
Assume that $\Lambda(\beta)=(\mu_1,\dots, \mu_d)$ for some $\mu_1,\dots, \mu_d$ such that $\mu_{d-1}=\mu_d$. Since $K\subset  J_{\varepsilon}$, then $\lim_{n}\alpha_{m_n}$ belongs to the domain of $\Pi$. Moreover, by \eqref{eq:yFalphabeta}, we have that 
\begin{align}\label{eq:PilimitinEmu}
\Pi_{*,j}(\lim_n\alpha_{m_n})\in \textbf{E}_{\mu_{j}}^{\hat{y}}\ \ \ \ \ \  \text{for all }\ \ 1\leq j\leq d-2.
\end{align}
On the other hand, since $\Lambda(\beta)\in B_{r}(D)$, we have that $\mu_1>\mu_2>\dots>\mu_{d-1}$, and consequently, $\textbf{E}_{\mu_{j}}^{\hat{y}}$ is one-dimensional for $1\leq j\leq d-2$. Therefore, using \eqref{eq:PilimitinEmu} as well as the fact that $|\Pi_{*,j}(\lim_{n}\alpha_{m_n})|=1$ for all $1\leq j\leq d$, it follows that
\begin{align}\label{eq:Piliminvi}
\Pi_{*,j}(\lim_{n}\alpha_{m_n})\in \{\Pi_{*,j}(\alpha),-\Pi_{*,j}(\alpha)\},
\end{align}
for all $1\leq j\leq d-2.$ Since the image of $\Pi_{*,j}$ is contained in $B_{\frac{1}{2}}(\Pi_{*,j}(\alpha))$, we conclude that $\Pi_{*,j}(\lim_{n}\alpha_{m_n})=\Pi_{*,j}(\alpha)$, which implies that $\varphi(\lim_{n}\alpha_{m_n})=\varphi(\alpha)$.  Therefore, using the fact that $\varphi$ is a diffeomorphism, we conclude that $\lim_{n}\alpha_{m_{n}}=\alpha$, as required. 

\medskip
\noindent \textit{Step 2}.
Next we prove that $DF_{p}$ is injective for all $p\in J_{\varepsilon}$. Consider an element $(a,b)\in \R^{\frac{d(d-1)}{2}-1}\times \R^{d-1}$ satisfying $DF_{\hat{x}_0}(a,b)=0$. Then, for $\varepsilon>0$ sufficiently small, the curve $M:(-\varepsilon,\varepsilon)\rightarrow \Sc_{deg}\cap B_{\delta}(\hat{x}_0)$ given by $M(t):=F(ta,tb)$, satisfies $M(0)=\hat{x}_0$ and $\dot{M}(0)=DF_{\hat{x}_0}(a,b)=0$. Denote by $v^{1}(t),\dots,v^{d}(t)$ the columns of $\Pi(ta)$ and define $\mu_{i}(t):=\Lambda_{i,i}(tb)$. Then, we have 
\begin{align}\label{eq:dyneigen}
M(t)v^i(t)
  &=\mu_{i}(t)v^i(t).
\end{align}
By taking derivative with respect to $t$ in \eqref{eq:dyneigen}, we get 
\begin{align*}
\dot M(t) v^i(t) +M(t) \dot v^i (t)
  &=\dot \mu_{i}(t)v^i(t) +\mu_{i}(t) \dot v^i(t),\ \ \ \ \ \ \ \ \ \text{for all}\ 1\leq i\leq d,
\end{align*}
which, by the condition $\dot{M}(0)=0$, implies that
\begin{align}\label{eq:Hadamar}
M(0) \dot v^i (0)
  &=\dot \mu_{i}(0) v^i(0) +\mu_{i}(0) \dot v^i(0),\ \ \ \ \ \ \ \ \ \text{for all}\ 1\leq i\leq d.
\end{align}
By taking the  inner product with $v^{j}(0)$ in \eqref{eq:Hadamar}, for $j\neq i$, we get
\begin{align*}
\Ip{v^{j}(0),\dot v^{i}(0)}(\mu_{j}(0)-\mu_{i}(0))=0.
\end{align*}
In particular, since $\mu_{d-1}=\mu_d$ is the only repeated eigenvalue for $\hat{x}_{0}$,
we deduce that for $1\leq i,j\leq d-1$ satisfying $i\neq j$, 
\begin{align}\label{eq:vdotivdotj}
\Ip{v^{j}(0),\dot v^{i}(0)}=0.
\end{align}
On the other hand, the condition $\Norm{v^{i}(t)}^2=1$ implies that 
\begin{align}\label{eq:vdotv}
\Ip{\dot{v}^{i}(0),v^{i}(0)}=0, 
\end{align}
which by \eqref{eq:vdotivdotj} leads to $\dot{v}^{i}(0)=0$ for all $1\leq i\leq d-1$. Since the last two columns of $\Pi$ are smooth functions of the first $d-2$ (see equations \eqref{eq:psi1def} and \eqref{eq:psi2def}), from the equations $\dot{v}^{1}(0)=\dots=\dot{v}^{d-1}(0)=0$, we conclude that $\frac{d}{dt}\Pi(ta)\big|_{t=0}=0$. On the  other hand, since $\Pi$ is a local chart for the manifold $\Oc(d;2)$ around $\Pi(0)$, the derivative $\dot{\Pi}(0)$ is injective, and thus the equation $\frac{d}{dt}\Pi(ta)\big|_{t=0}=0$ implies that $a=0$.

 Finally we prove that $b=0$. By definition, $M(t)=\Pi(\alpha t)\Lambda(\beta t)\Pi(\alpha t)^*$, and hence 
\begin{align*}
\dot{M}(t)
  &=\big(\frac{d}{dt}\Pi(\alpha t)\big)\Lambda(\beta t)\Pi(\alpha t)^*
	+\Pi(\alpha t)\frac{d}{dt}\Lambda(\beta t)\Pi(\alpha t)^*
	+\Pi(\alpha t)\Lambda(\beta t)\big(\frac{d}{dt}\Pi(\alpha t)\big).
\end{align*}
Since $a=0$, by evaluating the previous identity at $t=0$, we get 
 \begin{align*}
0
  &=\Pi(0)(\dot{\Lambda}(0)\beta)\Pi(0)^*,
\end{align*}
which implies that $b=0$. From here we conclude that the only solution to $DF_{x_{0}}(a,b)=0$ is $(a,b)=0$. This finishes the proof of the injectivity for $DF_{x_{0}}$.
The proof is now complete.
\end{proof}
The next result gives a sufficient condition for points $x_{0}\in\Hc_{deg}$ to have a neighborhood diffeomorphic to $\R^{n_{2}(d)-3}$.
\begin{Proposition}\label{lem:geometricaux2GUE}
Let $x_0\in \mathcal{H}_{deg}$ be such that $|\textbf{Sp}(\hat{x}_0)|=d-1$. Then, there exists $\gamma>0$, such that $\Hc_{deg}^d\cap B_{\gamma}(x_0)$ is an $(n_2(d)-3)$-dimensional manifold.
\end{Proposition}

\begin{proof}
\noindent Let $P\in\mathcal{H}(d)$ and $D\in\mathcal{D}(d)$ be such that 
\begin{align*}
\hat{x}_{0}
  &=PDP^*.
\end{align*}
Since $|\textbf{Sp}(\hat{x}_0)|=d-1$, only one of the eigenvalues $D_{1,1},\dots, D_{d,d}$ of $\hat{x}_{0}$ is repeated. We will assume without loss of generality that $D_{d-1,d-1}=D_{d,d}$. Define $\widetilde{J}_{\varepsilon}$, for $\varepsilon>0$, by $\widetilde{J}_{\varepsilon}:=(-\varepsilon,\varepsilon)^{d^2-d-2}$, and let $R\in \mathcal{U}(d;2)$ be the matrix $R=\{R_{i,j} ; 1\leq i\leq d,\ \  1\leq j\leq d-2$\}, with $R_{i,j}=P_{i,j}$ for all $1\leq i\leq d$ and $1\leq j\leq d-2$. Using the fact that for $\nu>0$ sufficiently small the set $\mathcal{V}_\nu^{R}$ given by \eqref{eq:Vtilde} is a manifold, we deduce that there exist $\varepsilon,\gamma>0$ and a diffeomorphism 
\begin{align*}
\widetilde{\varphi}:\widetilde{J}_{\varepsilon}\rightarrow \mathcal{V}_{\gamma}^R,
\end{align*}
such that $\widetilde{\varphi}(0)=R.$ As in the proof of Proposition \ref{lem:geometricaux2}, we can construct a smooth function $\Pi:\widetilde{J}_{\varepsilon}\rightarrow \mathcal{U}(d)$ with entries $\Pi_{i,j}$, such that $\Pi_{i,j}(\alpha)=\widetilde{\varphi}_{i,j}(\alpha)$ for all $\alpha\in\widetilde{J}_{\varepsilon}$ and $1\leq i\leq d$ and $1\leq j\leq d-2$.

 Define $\widetilde{F}:\widetilde{J}_{\varepsilon}\times\R^{d-1}\rightarrow\mathcal{H}_{deg}^d$ by
\begin{align*}
\widetilde{F}(\alpha,\beta):=\widetilde{\Pi}(\alpha)\Lambda(\beta)\widetilde{\Pi}(\alpha)^*.
\end{align*}
By Lemma \ref{lem:eigenprojectionsu}, there exists $\delta>0$ such that for all $x\in\mathcal{H}_{deg}^d\cap B_{\delta}(\hat{x}_0)$, there exist $Q\in\mathcal{U}(d)$ and $\Delta\in\mathcal{D}(d)$, satisfying 
\begin{align}\label{eq:hatxQDQ}
\hat{x}=Q\Delta Q^*,
\end{align} 
as well as
\begin{align*}
Q\in \mathcal{U}(d)\cap B_{\gamma/2}(P)\ \ \ \text{ and }\ \ \ \Delta\in \mathcal{D}(d)\cap B_{r}(\Delta),
\end{align*}
where $r$ is given by \eqref{eq:rdef}. Notice that relation \eqref{eq:hatxQDQ} still holds if we multiply the $j$-th column of $Q$, for $1\leq j\leq d-2$, by  $\langle P_{*,j},R_{*,j}\rangle/|\langle P_{*,j},R_{*,j}\rangle|$, so we can assume without loss of generality that 
$[Q_{*,1},\dots, Q_{*,d-2}]\in \mathcal{V}_{\gamma}^{R}.$ In particular, there exists $\alpha\in\widetilde{J}_{\varepsilon}$ such that $\widetilde{\varphi}(\alpha)=[Q_{*,1},\dots, Q_{*,d-2}]$. Then, by proceeding as in the proof of Proposition \ref{lem:geometricaux2}, we can show that 
$$\hat{x}=\widetilde{\Pi}(\alpha)\Lambda(\beta)\widetilde{\Pi}(\alpha)^*$$
for some $\beta\in\R^{d-1}$. As a consequence, if we define $\widetilde{V}:=B_{\delta}(\hat{x}_{0})$ and $\widetilde{U}:=F^{-1}(\widetilde{V})$, then the mapping 
\begin{align}\label{eq:diffeoFtilde}
\begin{array}{lll}\widetilde{U}\times\R^{d-1}&\rightarrow&\widetilde{V}\\ (\alpha,\beta)&\longmapsto&\widetilde{F}(\alpha,\beta)\end{array}
\end{align}
is onto. As in the proof of Proposition \ref{lem:geometricaux2}, provided that we show the conditions 
\begin{enumerate}
\item[(ii)] $\widetilde{F}^{-1}$ is continuous over $\tilde{U}$
\item[(iii)] $D\widetilde{F}_{p}$ is injective for every $\tilde{J}_{\varepsilon}$,
\end{enumerate}
then the mapping \eqref{eq:diffeoFtilde} is a diffeomorphism. The proof of the continuity of $\widetilde{F}^{-1}$ follows ideas similar to those from the GOE case. The only argument that needs to be modified is the proof of \eqref{eq:limalphasubseq}, since 
equation \eqref{eq:Piliminvi} is not necessarily true when $\beta=2$. To fix this problem, we replace equation \eqref{eq:Piliminvi} by
\begin{align*}
\widetilde{\Pi}_{*,i}(\lim_{n}\alpha_{m_n})=\eta \widetilde{\Pi}_{*,i}(\alpha), \ \ \ \ \text{for }\ \ 1\leq i\leq d-2,
\end{align*}	
which holds for some $\eta\in\C$ with $|\eta|=1$. Then, by using the fact that $[\Pi_{*,1}(\alpha),\dots, \Pi_{*,d-2}(\alpha)]$ belongs to $\mathcal{V}_{\gamma}^R$, we conclude that $\widetilde{\Pi}(\lim_{n}\alpha_{m_n})=\widetilde{\Pi}(\alpha)$, which in turn implies that $\varphi(\lim_{n}\alpha_{m_n})=\varphi(\alpha)$. Then, since $\varphi$ is a diffeomorphism we conclude that $\lim_{n}\alpha_{m_n}=\alpha$, as required.

The proof of the injectivity of $DF_{p}$ , for $p\in\widetilde{J}_{\varepsilon}$, follows the same arguments as in the GOE case, with the exception that identity \eqref{eq:vdotv} must be replaced by 
\begin{align}\label{eq:vdotvreal}
{\rm Re}(\Ip{\dot{v}^{i}(t),v^{i}(t)})=0.
\end{align}
Then, since $\Ip{v^{i}(t),v^{i}(0)}=|\Ip{v^{i}(t),v^{i}(0)}|$, we conclude that $\Ip{v^{i}(t),v^{i}(0)}_{\C^d}$ is real. This relation can be combined with \eqref{eq:vdotvreal}, in order to get \eqref{eq:vdotv}. The rest of the proof is analogous  to  Proposition \ref{lem:geometricaux2}.
\end{proof}

\section{Proof of the main results}\label{sec:proofs}
 This section is devoted to the proofs of Theorem \ref{thm:main} and Corollary \ref{cor:main}.

\begin{proof}[Proof of Theorem \ref{thm:main}]
The cases $\beta=1$ and $\beta=2$ can be handled similarly, so it suffices to prove the result for $\beta=1$. First suppose that $Q<2$. By Proposition \ref{lem:geometricaux1}, there exists an infinitely differentiable mapping $F:\R^{n_{1}(d)-2}\rightarrow\Sc(d)$, such that 
$\Sc_{deg}^{d}-A^1\subset {\rm Im}(F)$. As a consequence,  
\begin{multline}\label{eq:Pbcollisionupper}
\Pb\left[\lambda_{i}^{1}(t)=\lambda_{j}^{1}(t)\ \ \text{for some }\ t\in I\ \text{and}\ 1\leq i<j\leq d\right]\\
\begin{aligned}
  &=\Pb\left[X^1(t)\in \Sc_{deg}^d-A^1\ \ \text{for some }\ t\in I\right]\\
	&\leq\Pb\left[X^1(t)\in {\rm Im}(F)\ \ \text{for some }\ t\in I\right].
\end{aligned}
\end{multline}
Since the smooth mapping $F$ is defined over $\R^{n_1(d)-2}$, it follows that the set ${\rm Im}(F)$ has Hausdorff dimension at most $n_1(d)-2$. Thus, since $Q<2$, by Corollary \ref{cor:BiLaXi}, 
$$\Pb\left[X^1(t)\in {\rm Im}(F)\ \ \text{for some }\ t\in I\right]=0.$$ 
Therefore, by \eqref{eq:Pbcollisionupper} we get that 
\begin{align*}
\Pb\left[\lambda_{i}^1(t)=\lambda_{j}^1(t)\ \ \text{for some }\ t\in I\ \text{and}\ 1\leq i<j\leq n\right]
  =0,
\end{align*}
as required. To prove \eqref{eq:main100} in the case $Q>2$, choose any $x_{0}\in\Sc_{deg}^d$ satisfying $|\textbf{Sp}(\hat{x}_0)|=d-1$. By Lemma \ref{lem:geometricaux2}, there exists $\delta>0$, such that $\Sc_{deg}^d\cap B_{\delta}(x_{0})$ is an $n_1(d)$-dimensional manifold. In particular, the Hausdorff dimension of $\Sc_{deg}^d$ is at least $n_1(d)-2$.  The Hausdorff dimension of the shifted manifold $\Sc_{deg}^d-A^2$ is also larger than or equal to $n_1(d)-2$. Relation \eqref{eq:main100} then follows by Corollary \ref{cor:BiLaXi}. This finishes the proof of Theorem \ref{thm:main}.
\end{proof}

\begin{proof}[Proof of Corollary \ref{cor:main}]
The cases $\beta=1$ and $\beta=2$ can be handled similarly, so we will assume without loss of generality that $\beta=1$. 
Suppose that the process $\xi$ is a one dimensional fractional Brownian motion of Hurst parameter $0<H<1$, with $H\neq\frac{1}{2}$. If $H>\frac{1}{2}$, relation \eqref{eq:main10} follows from equation \eqref{eq:main100i} in Theorem \ref{thm:main}. Moreover, if $H<\frac{1}{2}$, then relation \eqref{eq:main1} follows from equation \eqref{eq:main100}. Therefore, it suffices to show relation \eqref{eq:fbmprobone} in the case where $H<\frac{1}{2}$ and $A^1\in \Sc_{deg}^d$ satisfies either $|\textbf{Sp}(A^1)|=d-1$ or $A^1=0$. The proof of this fact  will be done in two steps: first we prove that the probability of instant collision is strictly positive, and then we prove that such probability is one.

\medskip
\noindent
{\it Step 1.}  Our first goal is to prove that there exists $\delta'>0$ such that for any $0<T<1$,
\begin{align}\label{eq:Pbhittingprobfinalprev}
\Pb\left[\lambda^1_{i}(t)=\lambda^1_{j}(t)\ \ \text{for some }\ t\in (0,T]\ \text{and}\ 1\leq i<j\leq n\right]
  &\geq\delta^{\prime}>0.
  \end{align}
  We will split the proof of  (\ref{eq:Pbhittingprobfinalprev}) into  the cases $A^1=0$ and $|\textbf{Sp}(A^1)|=d-1$. 

\medskip
\noindent
{\it (i)}  Suppose $|\textbf{Sp}(A^1)|=d-1$. Then $A^1$ has exactly one repeated eigenvalue. We will assume without loss of generality that $\Phi_{d-1}(A^1)=\Phi_{d}(A^1)$. Fix $T<1$.  By  the self-similarity of $X^{1}(t)$, we can write
\begin{multline}\label{eq:hittingpselfs}
\Pb\left[\lambda^1_{i}(t)=\lambda^1_{j}(t)\ \ \text{for some }\ t\in (0,T]\ \text{and}\ 1\leq i<j\leq n\right]\\
\begin{aligned}
  &=\Pb\left[X^{1}(t)\in (\Sc_{deg}^d-A^1)\ \ \text{for some }\ t\in (0,T]\ \text{and}\ 1\leq i<j\leq n\right]\\
	&=\Pb\left[X^{1}(s)\in T^{-H}(\Sc_{deg}^d-A^1)\ \ \text{for some }\ s\in (0,1]\ \text{and}\ 1\leq i<j\leq n\right]\\
	&\geq\Pb\left[X^{1}(s)\in T^{-H}(\Sc_{deg}^d-A^1)\ \ \text{for some }\ s\in (1/2,1]\ \text{and}\ 1\leq i<j\leq n\right].
\end{aligned}
\end{multline} 
By Theorem \ref{BiLaXicapacity}, there exists $c_1>0$, such that 
\begin{multline}\label{eq:hittcapacitybound}
\Pb\left[X(s)\in T^{-H}(\Sc_{deg}^d-A^1)\ \ \text{for some }\ s\in (1/2,1]\ \text{and}\ 1\leq i<j\leq n\right]\\
  \geq c_1\Cc_{n_1(d)-\frac 1H }(T^{-H}(\Sc_{deg}^d-A^1)).
\end{multline}
  
Let $G:(-1,1)^{n_1(d)-2}\rightarrow\Sc ^d_{deg} -A^1$ be a parametrization of the manifold $\Sc_{deg}^d-A^1$ around zero. Consider the probability measure $m_{\varepsilon}(dx):=(2\varepsilon)^{2-n_1(d)}\Indi{[-\varepsilon,\varepsilon]^{n_1(d)-2}}(x)dx$  and let $\nu_{\varepsilon}(dx)$ be the pullback measure of $m_{\varepsilon}$ under the map $x\mapsto \varepsilon^{-1}G(x)$. Define $f_{\alpha}$ by \eqref{eq:falphadef}. Since $\nu_{T^H}(dx)$ is a probability measure with support in $T^{-H}(\Sc_{deg}^d-A^1)$, we have 
\begin{align}\label{eq:capacitybound}
\Cc_{n_1(d)-\frac 1H }(T^{-H}(\Sc_{deg}^d-A^1))
  &\geq \bigg(\int_{T^{-H}(\mathcal{S}_{deg}^d-A^1)}f_{n_1(d)-\frac 1H}(\|u-v\|)\nu_{T^{H}}(du)\nu_{T^{H}}(dv)\bigg)^{-1}\nonumber\\
	&\geq \bigg((2T^{H})^{2(2-n_{1}(d))}\int_{(-T^{H},T^{H})^{2(n_1(d)-2)}}f_{n_1(d)-\frac 1H}(T^{-H}\|G(x)-G(y)\|  )dxdy\bigg)^{-1}\nonumber\\
	&= 2^{2(n_{1}(d)-2)}\bigg(\int_{(-1,1)^{2(n_1(d)-2)}}f_{n_1(d)-\frac 1H}(T^{-H}\|G(T^{H}x)-G(T^{H}y)\|)dxdy\bigg)^{-1}.
\end{align}
By the mean value theorem, there exists $\tau\in(0,1)$, depending on $T$, such that the vector $v(\tau):=\tau(1-x)+\tau y$ satisfies
\begin{align}\label{eq:DGvmean}
T^{-H}(G(T^{H}x)-G(T^{H}y))
  &=T^{-H}\frac{d}{d\tau}G(T^{H}(\tau(1-x)+\tau y))=DG_{v(\tau)}[x-y].
\end{align}
Consider the mapping $(w,v)\mapsto\|DG_{v}[w]\|$, defined over the compact set $K:=\{(w,v) : \|w\|=1,\ \ \text{ and }\ v\in[-T^{H},T^{H}]^{n_1(d)-2}\}$. By the smoothness of $G$, this mapping has a minimizer $(w_{0},\tau_0)$. Moreover, since $DG_v$ is injective for $v$ near zero, we have that $\delta:=\|DG_{v_0}[w_0]\|>0$. Using this observation as well as relation \eqref{eq:DGvmean}, we get that 
\begin{align*}
T^{-H}\|G(T^{H}x)-G(T^{H}y)\|
  &=\|x-y\| \|DG_{v(\tau)}[\|x-y\|^{-1}(x-y)]\|\\
	&\geq\delta\|y-x\|.
\end{align*}
Therefore, by \eqref{eq:capacitybound}, it follows that if $n_1(d)>\frac 1H$, 
\begin{align*}
\Cc_{n_1(d)-\frac 1H}(T^{-H}(\Sc_{deg}^d-A^1))
  &\geq  (2\delta)^{2(n_1(d)-\frac{1}{H})}\bigg(\int_{(-1,1)^{2(n_1(d)-2)}}\|x-y\|^{\frac 1H-n_1(d)}dxdy\bigg)^{-1}.
\end{align*}
The integral in the right-hand side is finite due to the condition $\frac 1H>2$, and thus,  there exists a constant $\delta^{\prime}>0$, such that 
\begin{align}\label{ineq:Capdeltaprime}
\Cc_{n_1(d)-\frac 1H}(T^{-H}(\Sc_{deg}^d-A^1))\geq\delta^{\prime}>0.
\end{align}
By following a similar approach, we can show that \eqref{ineq:Capdeltaprime} also holds for the case $n_1(d)=\frac 1H$, while in the case $n_{1}(d)<\frac 1H$,  identity \eqref{ineq:Capdeltaprime} follows from the fact that $f_{\alpha}=1$ for all $\alpha>0$. Combining \eqref{eq:hittingpselfs}, \eqref{eq:hittcapacitybound} and \eqref{ineq:Capdeltaprime}, we conclude that there exists $\delta^{\prime}>0$ such that for all $T\in(0,1)$,
(\ref{eq:Pbhittingprobfinalprev}) holds.

\medskip
\noindent
{\it (ii)}  Next we show that relation \eqref{eq:Pbhittingprobfinalprev} holds as well in the case $A=0$, if $\delta^{\prime}>0$ is sufficiently small. Notice that if $A=0$, by the self-similarity of  $\xi$ and the homogenity of the function $(\Phi_{1},\dots, \Phi_{d})$,  we have
\begin{multline*}
\Pb\left[\lambda^1_{i}(t)=\lambda^1_{j}(t)\ \ \text{for some }\ t\in (0,T]\ \text{and}\ 1\leq i<j\leq n\right]\\
\begin{aligned}
  &=\Pb\left[\Phi_{i}(X^{1}(t))=\Phi_{j}(X^{1}(t))\ \ \text{for some }\ t\in (0,T]\ \text{and}\ 1\leq i<j\leq n\right]\\
	&=\Pb\left[\Phi_{i}(T^{H}X^{1}(t))=\Phi_{j}(T^HX^{1}(t))\ \ \text{for some }\ t\in (0,1]\ \text{and}\ 1\leq i<j\leq n\right]\\
	&=\Pb\left[\lambda^1_{i}(t)=\lambda^1_{j}(t)\ \ \text{for some }\ t\in (0,1]\ \text{and}\ 1\leq i<j\leq n\right]\\
	&\geq \Pb\left[\lambda^1_{i}(t)=\lambda^1_{j}(t)\ \ \text{for some }\ t\in [1/2,1]\ \text{and}\ 1\leq i<j\leq n\right].
\end{aligned}
\end{multline*}
Relation \eqref{eq:Pbhittingprobfinalprev} for $A=0$ then follows from Theorem \ref{thm:main}.

\medskip
\noindent
{\it Step 2.} 
By taking $T\rightarrow0$ in the left hand side of \eqref{eq:Pbhittingprobfinalprev}, we get
\begin{align}\label{eq:Pbhittingprobfinal}
\Pb\left[\bigcap_{T\in(0,1)}\{\lambda^1_{i}(t)=\lambda^1_{j}(t)\ \ \text{for some }\ t\in (0,T]\ \text{and}\ 1\leq i<j\leq n\}\right]
  &\geq\delta^{\prime}>0.
\end{align}
Finally, for $i\leq j$, we write $\xi_{i,j}$ as a Volterra process of the form $\xi_{i,j}(t)=\int_{0}^{t}K_{H}(s,t)dW_{i,j}(t)$, where the  $\{W_{i,j}(t) ; t\geq0\}$  are independent  standard Brownian motions and 
\begin{align*}
K_{H}(s,t)
  &:=c_{H}\bigg(\big(t/s\big)^{H-\frac{1}{2}}(t-s)^{H-\frac{1}{2}}-(H-1/2)s^{\frac{1}{2}-H}\int_{s}^tu^{H-\frac{3}{2}}(u-s)^{H-\frac{1}{2}}du\bigg),
\end{align*}
where $c_{H}:=(2H)^{-\frac{1}{2}}(1-2H)\int_{0}^1(1-x)^{-2H}x^{H-\frac{1}{2}}dx$. We can easily check that
\begin{align*}
\bigcap_{T\in(0,1)}\{\lambda^1_{i}(t)=\lambda^1_{j}(t)\ \ \text{for some }\ t\in (0,T]\ \text{and}\ 1\leq i<j\leq n\}
\end{align*}
belongs to the germ $\sigma$-algebra $\mathcal{F}_{0+}:=\bigcap_{s>0}\sigma\{W_{i,j}(u) ; 0<u\leq s, 0\le i\le j \le d\}$.	Thus, combining \eqref{eq:Pbhittingprobfinal} with  Blumenthal's zero-one law, we conclude that
\begin{align*}
\Pb\left[\bigcap_{T\in(0,1)}\{\lambda^1_{i}(t)=\lambda^1_{j}(t)\ \ \text{for some }\ t\in (0,T]\ \text{and}\ 1\leq i<j\leq n\}\right]
  &=1.
\end{align*}
The proof is now complete.
\end{proof}


\end{document}